\documentclass{birkjour}

\usepackage{epsfig,enumerate,amsmath,amsfonts,amssymb}
\usepackage{indentfirst}

\usepackage{pstricks}
\usepackage{latexsym}

\newtheorem{thm}{Theorem}[section]
\newtheorem{cor}[thm]{Corollary}
\newtheorem{lem}[thm]{Lemma}
\newtheorem{prop}[thm]{Proposition}
\theoremstyle{definition}
\newtheorem{defn}[thm]{Definition}
\theoremstyle{remark}
\newtheorem{rem}[thm]{Remark}

\numberwithin{equation}{section}

\def\R{\mathbb R}
\def\N{\mathbb N}
\newcommand{\pa} {\partial}
\newcommand{\al} {\alpha}
\newcommand{\Om} {\Omega}

\newcommand{\la} {\lambda}
\newcommand{\La} {\Lambda}
\newcommand{\noi} {\noindent}
\newcommand{\ra} {\rightarrow}
\newcommand{\var} {\varepsilon}
\newcommand{\WW} {W^{1,p(x)}_0(\Om)}
\newcommand{\DD}  {\displaystyle}
\newcommand{\TT} {\Delta_t}
\newcommand{\X}{\mathbb W}
\newcommand{\ii}{\int_0^T\!\!\!\int_\Om}
\newcommand{\eqdef}{\stackrel{{\rm{def}}}{=}}
\begin{document}

%
%
%
%
%
%
%
%
%

\title[Quasilinear parabolic problem involving the $p(x)$-Laplacian]{Quasilinear parabolic problem with\\ $p(x)$-Laplacian: existence, uniqueness\\ of weak solutions and stabilization}

\author[J. Giacomoni]{Jacques Giacomoni}

\address{Laboratoire de Math\'ematiques et leurs applications Pau, UMR CNRS 5142,\\
Universit\'e de Pau et Pays de l'Adour,\\ Avenue de l'Universit\'e, BP 1155, 64013 Pau Cedex,\\
 France}

\email{jacques.giacomoni@univ-pau.fr}

\author[S. Tiwari]{Sweta Tiwari}

\address{Instituto de Matematicas,\\
 Universidad Nacional Aut\'onoma de M\'exico,\\
  Circuito exterior, Ciudad Universitaria, 04510, M\'exico, D.F.}

\email{swetatiwari.iitd@gmail.com}

\author[G. Warnault]{Guillaume Warnault}

\address{Laboratoire de Math\'ematiques et leurs applications Pau, UMR CNRS 5142,\\ Universit\'e de Pau et Pays de l'Adour,\\ Avenue de l'Universit\'e, BP 1155, 64013 Pau Cedex,\\ France}

\email{guillaume.warnault@univ-pau.fr}



\date{June 9, 2015}

\begin{abstract}
We discuss the existence and uniqueness of the weak solution
of the following quasilinear parabolic equation 
\begin{equation*}
\left\{
\begin{array}{ll}
u_t-\Delta _{p(x)}u = f(x,u)&\text{in } Q_T\eqdef (0,T)\times\Om, \\
u  = 0 &\text{on  }  \Sigma_T\eqdef (0,T)\times\pa\Om,\\
u(0,x)=u_0(x)&\text{in }\Om
\end{array}\right. \tag{$P_T$} \end{equation*}
involving the $p(x)$-Laplacian operator. Next, we discuss the global behaviour of solutions and in particular some stabilization properties.
\end{abstract}

\maketitle

\section{Introduction}
\noindent Our main goal in this paper is to study the existence, uniqueness and global behaviour of the weak solutions to the following parabolic equation involving the $p(x)$-Laplacian operator
\begin{equation*}\label{Pt}
\left\{
\begin{array}{ll}
u_t-\Delta _{p(x)}u = f(x,u)&\text{in } Q_T= (0,T)\times\Om, \\
u  = 0 &\text{on  }  \Sigma_T= (0,T)\times\pa\Om,\\
u(0,x)=u_0(x)&\text{in }\Om
\end{array}\right. \tag{$P_T$}\end{equation*}
where $\Om\subset\R^d$, $d\geq 2$, is a smooth bounded domain, $p:\Om \ra [1,+\infty]$ and  $f: \Omega\times \R\rightarrow \R$ is a Caratheodory function. \\
Let $\mathcal P(\Om)$ be the family of all measurable functions $q:\Omega \ra [1,+\infty]$ and set
$$\mathcal P^{log}(\Om)\eqdef \left\{q\in\mathcal P(\Om)\,:\, \frac{1}{q} \mbox{ globally log-H\"older continuous}\right\}.$$
In particular, for any $p\in\mathcal P^{\log}(\Om)$, there exists a function $w$ such that 
$$\forall (x,y)\in \Omega^2,\ |p(x)-p(y)|\leq w(|x-y|) \ \mbox{and} \ \limsup_{t\ra0^+} \,-w(t)\ln t <+\infty.$$
There is an abundant litterature devoted to questions on existence and \break uniqueness of solutions to $(P_T)$ for $p(x)\equiv p$ (see for instance \cite{BBG} and references
therein).
More recently, parabolic and elliptic problems with variable exponents have been studied
quite extensively, see for example \cite{AMS,AS1,AS2,AS3,CLR,MR,R}.
The importance of investigating these problems lies in their occurrence in modeling 
various physical problems involving strong anisotropic phenomena related to
electrorheological fluids (an important class of non-Newtonian fluids)
\cite{AMS,RR,R}, image processing \cite{CLR}, elasticity \cite{Zhikov}, the processes of filtration in complex media \cite{AntShm}, stratigraphy problems \cite{GV} and also mathematical biology \cite{Frag}. 

Regarding the current litterature, we bring in this paper new results about the regularity of weak solutions and about the behaviour of global weak solutions. In particular, we investigate the question of asymptotic convergence to a steady state. To prove the existence of weak solutions,  we follow a semi-group approach, involving a semi-discretization in time method,  that provides the existence of mild solutions belonging to $C([0,T];\X)$ and $C([0,T];L^\infty(\Omega))$. Then the existence of  subsolutions and supersolutions and the weak comparison principle reveal the stabilization property for a suitable class of nonlinearities  $f$. 

In our knowledge, the existence of mild solutions and the convergence to a stationary solution for quasilinear parabolic equations with variable exponents were not investigated previously in the litterature and all the corresponding results brought in the present paper are new. To establish these results, we use some former contributions about the validity of a strong comparison principle (see \cite{Zh}), the regularity of solutions (see in particular \cite{AMS,Fa0,FS}) and some extensions proved in appendix C. We point out that other aspects of global behaviour of weak solutions (extinction in finite time, localization, blow-up in finite time) are discussed in \cite{AS2,AS3}. In \cite{AS1}, the Galerkin method is used alternatively to prove existence of weak solutions. In the same way, in the case where $f\equiv 0$ or $f$ depends only to $(t,x)\in Q_T$, the authors of \cite{AZ,AntShm} use a perturbation method to establish the existence of solutions of $(P_T)$.

Concerning quasilinear elliptic equations with variable exponents, striking results about existence  and non-exitence of eigenvalues contrasting with the constant exponent case are proved in \cite{MR,RS} showing the complex nature of the $p(x)$-laplacian operator. Furthermore, in \cite{RS}, the authors established also multiplicity results for combined concave-convex function $f$.

Before going further, we recall the definitions and useful results on the variable exponent Lebesgue and Sobolev spaces. For more details, we refer to the book \cite{DHHR} and the paper \cite{KR}.
Let $p\in \mathcal P(\Omega)$. We define the semimodular
$$\rho_p(u)\eqdef \int_{\Om\backslash \Omega_\infty}|u|^{p(x)}dx+ess\sup_{\Omega_\infty}|u(x)|$$
where $\Omega_\infty=\{x\in \Omega\,|\, p(x)=\infty\}$. Then the variable exponent Lebesgue space is defined as follows:
$$L^{p(x)}(\Om)\eqdef \left\{u\,|\,u\mbox{ is measurable on $\Om$ and } \rho_p(\lambda u)<\infty\ \mbox{ for some
}\lambda>0\right\}.$$
If $p\in L^\infty(\Omega)$, this definition is equivalent to (see Theorem 3.4.1 in \cite{DHHR})
$$L^{p(x)}(\Om)=\left\{u\,|\,u\mbox{ is measurable on $\Om$ and }\ \displaystyle\rho_p(u)<\infty\right\}.$$
This is a normed linear space equipped with the Luxemburg norm
$$\|u\|_{L^{p(x)}(\Om)}= \inf\left\{\la>0:\ \rho_p(\frac{u}{\lambda})\leq 1\right\}.$$
Define $p_{-}=\displaystyle\inf_{x\in\Om}p(x)$ and $p_{+}=\displaystyle\sup_{x\in\Om}p(x)$. Then
$$L^{p_+}(\Om)\subset L^{p(x)}(\Om) \subset L^{p_-}(\Om).$$
We denote by $p_c$ the conjugate exponent of $p$ defined as
$$p_c(x)=\frac{p(x)}{p(x)-1}.$$
We have the following well-known properties on $L^{p(x)}$ spaces (see \cite{KR}).
\begin{prop}\label{P1}
Let $p\in L^\infty(\Omega)$. Then for any $u\in L^{p(x)}(\Om)$ we have:
\begin{enumerate}[(i)]
\item $\rho_p(u/\|u\|_{L^{p(x)}(\Om)})=1$.
\item $\|u\|_{L^{p(x)}(\Om)}\ra 0$\; \text{if and only if}\;  $\rho_p(u)\ra 0$.
\item $L^{p_c(x)}(\Omega)$ is the dual space of $L^{p(x)}(\Om)$.
 \end{enumerate}
\end{prop}
\noindent Proposition \ref{P1} {\it (i)} implies that: if $\|u\|_{L^{p(x)}(\Om)}\geq 1$,
\begin{equation}\label{a}
\|u\|_{L^{p(x)}(\Om)}^{p_{-}}\leq \rho_p(u)\leq\|u\|_{L^{p(x)}(\Om)}^{p_{+}}
\end{equation}
and if $\|u\|_{L^{p(x)}(\Om)}\leq 1$
\begin{equation}\label{b}
\|u\|_{L^{p(x)}(\Om)}^{p_{+}}\leq \rho_p(u)\leq\|u\|_{L^{p(x)}(\Om)}^{p_{-}}.
\end{equation}
Moreover, we have also the generalized H\"older inequality: for $p\in \mathcal P(\Om)$, there exists a constant $C=C(p_+,p_-)\geq 1$ such that  for any $f\in L^{p(x)}(\Om)$ and $g\in L^{p_c(x)}(\Om)$
\begin{equation}\label{hi}
 \int_\Om |f(x)g(x)|dx\leq C\|f\|_{L^{p(x)}(\Om)}\|g\|_{L^{p_c(x)}(\Om)}.
 \end{equation}
\noindent The corresponding Sobolev space is defined as follows:
$$W^{1,p(x)}(\Om)\eqdef \left\{u\in L^{p(x)}(\Om):|\nabla u|\in L^{p(x)}(\Om)\right\}.$$
A natural norm defined on $W^{1,p(x)}(\Om)$ is
$$\|u\|_{W^{1,p(x)}}=\|u\|_{L^{p(x)}(\Om)}+\|\nabla u\|_{L^{p(x)}(\Om)}.$$
For the sake of convenience, we define $\X=\WW$, the closure of $C^\infty_0(\Om)$ in $W^{1,p(x)}(\Om)$ for $p\in \mathcal P^{log}(\Omega)\cap L^\infty(\Omega)$. Since the domain $\Om$ is a bounded domain, the Poincar\'e inequality holds and thus we define the norm on $\X$ as $\|u\|_\X=\|\nabla u\|_{L^{p(x)}(\Om)}$.\\
$L^{p(x)}(\Om)$, $W^{1,p(x)}(\Om)$ and $\X$ are  Banach spaces. Moreover they are separable if $p\in L^\infty(\Omega)$
and
reflexive if $1<p_-\leq p_+<\infty$ (see \cite{KR}).
Furthermore we have the following Sobolev embedding Theorem (see \cite{DHHR}):
\begin{thm}\label{SE1}
Let $p\in  \mathcal P^{log}(\Omega)$ satisfies $1\leq p_-\leq p_+<d$. Then,\break
$W^{1,p(x)}(\Om)\hookrightarrow L^{\al(x)}(\Om)$ for any $\al\in L^\infty(\Omega)$  such that for all $x\in \Om$, $\al(x)\leq p^*(x)=\frac{d\,p(x)}{ d-p(x)}.$
Also the previous embedding is compact for $\al(x)<p^*(x)-\varepsilon$ a.e. in $\Om$ for any $\varepsilon>0$.
\end{thm}
The paper is organized as follows. The next section (Section 2) contains the statements of our main results on the existence, uniqueness, regularity of solutions to $(P_T)$  (see Theorems \ref{PT1}, \ref{PT3}, \ref{continuity u with f}) and on the global behaviour of solutions (see Theorems \ref{PT3} and \ref{stabilization}). In Section 3, we deal with the existence of weak solutions to the auxilary problem $(S_T)$. Main results concerning the existence of weak solutions to $(P_T)$ are established in Section 4. Finally the existence of mild solutions and stabilization properties are proved in Section 5. The appendices A, B contain some technical lemmata about monotonicity and compactness properties of $p(x)$-Laplacian operator. In Appendix C, we establish some new regularity results ($L^\infty$-bound) about quasilinear elliptic equations involving $p(x)$-laplacian used in Sections 3-5.

\section{Main results}
\noindent In the rest of the paper, we assume that $p\in \mathcal P^{log}(\Om)$ such that 
$$\frac{2 d}{d+2}<p_{-}\leq p_+< d.$$
First we consider the following problem:
\begin{equation*}\label{St}
\left\{\begin{array}{ll}
u_t-\Delta _{p(x)}u = h(t,x) &\text{in } Q_T,  \\
 u  = 0 &\text{on  }
 \Sigma_T,\\
u(0,x)=u_0(x)&\text{in }\Om,
\end{array}\right.\tag{$S_T$}
 \end{equation*}
where $T>0$, $h\in L^2(Q_T)\cap L^q(Q_T)$, $q>\displaystyle \frac{d}{p_-}$. Considering the initial data in $u_0\in\X\cap L^\infty(\Om)$, we study the weak solution to $(S_T)$ defined as follows:
\begin{defn}\label{defi1}
A weak solution to $(S_T)$ is any function $u\in L^\infty(0,T;\X)$ such that $u_t\in L^2(Q_T)$ and satisfying for any $\phi\in
C^\infty_0(Q_T)$
\begin{equation*}
\ii u_t\phi\,dxdt+\ii|\nabla u|^{p(x)-2}\nabla u\cdot\nabla\phi\,dxdt=\ii
h(t,x)\phi\,dxdt
\end{equation*}
and $u(0,.)=u_0$ a.e. in $\Omega$.
\end{defn}
\noindent Similarly we define the weak solutions to the problem $(P_T)$ as follows:
\begin{defn}\label{ws2}
A solution to $(P_T)$ is a function $u\in L^\infty(0,T;\X)$ such that $u_t\in L^2(Q_T)$, $f(.,u)\in
L^\infty(0,T;L^2(\Omega))$ and for any $\phi\in C^\infty_0(Q_T)$
\begin{equation*}
\ii u_t\phi\,dxdt+\ii|\nabla u|^{p(x)-2}\nabla u\cdot\nabla\phi\,dxdt=\ii
f(x,u)\phi\,dxdt
\end{equation*}
and $u(0,.)=u_0$ a.e. in $\Omega$.
\end{defn}
\noindent Hence with the above definitions, we establish the following local existence results: 
\begin{thm}\label{sol for (S_t)}
Let $T>0$, $u_0\in \X\cap L^\infty(\Om)$ and $h\in L^2(Q_T)\cap L^q(Q_T)$, $q>\frac{d}{p_-}$. Then, $(S_T)$ admits a
unique solution $u$ in the sense of Definition \ref{defi1}. Moreover $u\in C([0,T];\X).$
\end{thm}

\begin{thm}\label{PT1}
Let $f: \Omega\times \R\rightarrow \R$ be a Caratheodory function satisfying the following two conditions:
\begin{enumerate}
\item[$\bf{(f_1)}$] $\ t\ra f(x,t)$ is locally Lipschitz uniformly in $x\in\Omega$;\medskip
\item[$\bf{(f_2)}$] there exists  $\alpha\in \R$  such that $x\rightarrow f(x,\alpha)\in L^2(\Om)\cap L^q(\Om)$, $q>\frac{d}{p_-}$.
\end{enumerate}
Assume in addition that one of the following hypotheses holds:
\begin{itemize}
\item[\bf{(H1)}] there exists a nondecreasing locally Lipschitz function $L_0$ such that
\begin{equation*}
  |f(x,v)|\leq  L_0(v), \quad a.e.\  (x,v)\in \Omega\times \R;
\end{equation*}
\item[\bf{(H2)}] there exist two
nondecreasing locally Lipschitz functions $L_1$ and $L_2$ such that
\begin{equation*}
 L_1(v)\leq f(x,v)\leq  L_2(v), \quad a.e.\ (x,v)\in \Omega\times \R.
\end{equation*}
\end{itemize}
Then, for any $u_0\in \X\cap L^\infty(\Omega)$,  there exists $\tilde T\in (0,+\infty]$ such that for any $T\in [0, \tilde T)$, $(P_T)$ admits a unique solution $u$ in sense of Definition \ref{ws2}. Moreover for
any
$r> 1$, $u\in C([0,T];L^r(\Om))\cap C([0,T];\X)$.
\end{thm}
Under additional hypothesis about the growth of $f$ and regularity of the intial data, we are able to prove the existence of global solutions. Precisely, we have the following result:
\begin{thm}\label{PT3}
Let $f$ be a Caratheodory function satisfying $\bf{(f_1)}$ and the additional condition:
\begin{enumerate}
\item[$\bf{(f_3)}$] there exists $C>0$ such that $\forall (x,s)\in\Omega\times\R$, $|f(x,s)|\leq C(1+|s|^\beta)$ where $\beta<p_{-}-1$.
\end{enumerate}
Assume in addition that one of the following conditions is valid:
\begin{itemize}
\item[\bf{(C1)}] $u_0\in\X$ such that $\Delta_{p(x)}u_0\in L^q(\Om)$ where $q> \frac{d}{p_-}$;
\item[\bf{(C2)}] $u_0\in C^1_0(\overline \Omega)$ and $p\in C^1(\overline{\Omega})$.
\end{itemize} 
Then, for any $T>0$, $(P_T)$ admits a unique weak solution in the sense of Definition \ref{ws2}. Moreover $u\in C([0,T];\X)$.
\end{thm}
\begin{rem}\ 
\begin{enumerate}
\item Theorem \ref{PT3} is still valid, under the condition {\bf (C2)}, replacing $\bf{(f_3)}$ by the hypotheses on $f$: 
\begin{enumerate}
\item[$\bf{(f_4)}$] there exists $\zeta\in \R$ such that $x\rightarrow f(x,\zeta)\in L^\infty(\Omega)$; 
\item[$\bf{(f_5)}$]$\lim_{|s|\rightarrow +\infty} \frac{|f(x,s)|}{|s|^{p_--1}}=0.$
\end{enumerate}
\item Under an additional asymptotic super homogeneous growth assumption on $f$ and for initial data large enough, blow up in finite time of solutions can also occur. For instance, let $f(x,v)=v^q$ with $q>p^+$ and define the energy functional
\begin{equation*}
E(u)\eqdef\int_\Omega\frac{\vert \nabla u\vert^{p(x)}}{p(x)}\,dx-\int_\Omega\frac{u^{q+1}}{q+1}\,dx.
\end{equation*}
Then, using a well-known energy method and for any initial data $u_0$ satisfying $E(u_0)<0$, the weak solution to $(P_T)$ blows up in finite time. For further discussions of global behaviour of solutions (blow up, localization of solutions, extinction of solutions) to quasilinear anisotropic  parabolic  equations involving variable exponents, we refer to \cite{AS2} and \cite{AS3}.
\end{enumerate}
\end{rem}
\noindent Next, we investigate the asymptotic behaviour of global solutions, in particular the convergence to a stationary solution. For that we appeal the theory of maximal accretive operators in Banach spaces (see Chapters 3 and 4 in \cite{Ba}) that provides the existence of mild solutions. Precisely, observing that the operator $A\eqdef-\Delta_{p(x)}$, with Dirichlet boundary conditions, is m-accretive in $L^\infty(\Omega)$ with
\begin{equation*}
{\mathcal D}(A)=\left\{u\in\X\cap L^\infty(\Omega)\,\vert\, Au\in L^\infty(\Omega)\right\}
\end{equation*} 
as the domain of the operator $A$,
we get the above results which essentially follow from Theorems \ref{sol for (S_t)} and \ref{PT1} with Theorem 4.2 (page 130) and Theorem 4.4 (page 141) in \cite{Ba}: 
\begin{thm}\label{continuity de u} Let $T>0$, $h\in L^\infty(Q_T)$ and let $u_0$ be in $\X\cap\overline{{\mathcal D}(A)}^{L^\infty}$. Then,
\begin{enumerate}
\item[(i)]the unique weak solution $u$ to $(S_T)$ belongs to $\mathcal C([0,T]; \mathcal C_0(\overline{\Omega}))$.
\item[(ii)] If $v$ is another mild solution to $(S_T)$  with the initial datum $v_0\in\X\cap\overline{{\mathcal D}(A)}^{L^\infty}$ and the right-hand side $k\in L^\infty(Q_T)$, then the following estimate holds:
\begin{equation}
\|u(t)-v(t)\|_{L^\infty(\Omega)}\leq \|u_0-v_0\|_{L^\infty(\Omega)}+\int_0^t\|h(s)-k(s)\|_{L^\infty(\Omega)}\, {\rm d}s,\quad 0\leq t\leq T.\label{estimation2}
\end{equation}
\item[(iii)] If $u_0\in {\mathcal D}(A)$ and $h\in W^{1,1}(0,T;L^\infty(\Omega))$ then $u\in W^{1,\infty}(0,T;L^\infty(\Omega))$ and $\Delta_{p(x)} u\in L^\infty(Q_T)$, and the following estimate holds:
\begin{equation}
\left \|\frac{\partial u}{\partial t}(t)\right \|_{L^\infty(\Omega)}\leq \|\Delta_{p(x)} u_0+h(0)\|_{L^\infty(\Omega)}+\int_0^T\left \|\frac{\partial h}{\partial t}(t)\right \|_{L^\infty(\Omega)}{\rm d}\tau.\label{estimationbiss2}
\end{equation}
\end{enumerate}
\end{thm}
Concerning problem $(P_T)$, we deduce the following similar result:
\begin{thm}\label{continuity u with f}
Assume that conditions and hypotheses on $f$ in Theorem \ref{PT1} are satisfied. Let $u_0\in \X\cap\overline{{\mathcal D}(A)}^{L^\infty}$. Then, the unique weak solution to $(P_T)$ 
belongs to $C([0,T];C_0(\overline{\Omega}))$ and 
\begin{enumerate}
\item[(i)] there exists $\omega>0$ such that if $v$ is another weak solution to $(P_T)$ with the initial datum $v_0\in \X\cap\overline{{\mathcal D}(A)}^{L^\infty}$ then the following estimate holds for $T<\tilde{T}$:
$$
\|u(t)-v(t)\|_{L^\infty(\Omega)}\leq e^{\omega t} \|u_0-v_0\|_{L^\infty(\Omega)},\quad 0\leq t\leq T.
$$
\item[(ii)]  If $u_0\in {\mathcal D}(A)$ then $u\in W^{1,\infty}(0,T;L^\infty(\Omega))$ and $\Delta_{p(x)} u\in L^\infty(Q_T)$, and the following estimate holds:
$$
\left \|\frac{\partial u}{\partial t}(t)\right \|_{L^\infty(\Omega)}\leq e^{\omega t} \|\Delta_{p(x)} u_0+f(x,u_0)\|_{L^\infty(\Omega)}.
$$
\end{enumerate}
\end{thm}
\begin{rem}\label{omega-acretivity}\ 
\begin{enumerate}
\item The constant $\omega$ in Theorem \ref{continuity u with f} is  the Lipschitz constant of $f$ in\break $[ v_1(T),v_2(T)]$ (respectively $[-v_0(T),v_0(T)]$) given in \eqref{CL}. If $f$ is nonincreasing with respect to the second variable and $x\to f(x,0)\in L^\infty(\Omega)$, $\omega=0$ can be taken in assertions (i) and (ii) above. In this case, note that $-\Delta_{p(x)}-f(x,\cdot)$ is m-accretive in $L^\infty(\Omega)$ (see proposition \ref{maccretive}).
\item If we assume hypotheses in Theorem \ref{PT3}, then the weak solution to $(P_T)$ belongs to $C([0,+\infty), C_0(\overline{\Omega}))$.
\end{enumerate}
\end{rem}
\noindent Using the above results, we give some  stabilization properties for $(P_T)$ for global solutions. Precisely, we prove the following:
\begin{thm}\label{stabilization}
Assume that $f$ satisfies $\bf{(f_1)}$, $\bf{(f_4)}$ and is nonincresing in respect to the second variable. Then, for any initial data $u_0\in C^1_0(\overline{\Omega})$, the weak solution, $u$, to $(P_T)$ is defined in $(0,\infty)\times \Omega$, belongs to $C([0,+\infty); C_0(\overline{\Omega}))$ and verifies
$$
u(t)\to u_\infty\quad\mbox{ in } L^\infty(\Omega)\quad\mbox{as }t\to\infty
$$
where $u_\infty$ is the unique stationary solution to $(P_T)$.
\end{thm}
\begin{rem}
\cite{Fa00} establish uniqueness results for quasilinear elliptic equations involving the $p(x)$-laplacian under different conditions on $f$ (see Theorem 1.2 for instance). Theorem \ref{stabilization} is still valid in this case.
\end{rem}


\section{Existence of solutions of $(S_T)$}
\noindent First, we consider the following quasilinear elliptic problem:
\begin{equation*}\label{(P)}
\Biggl\{\begin{array}{ll}
u-\la\Delta _{p(x)}u = g&\text{ in } ~\Om,  \\
u  = 0 ~&\text{  on   }~
 \pa\Om
\end{array}\tag{P}  \end{equation*}
with $\la>0$ and $g$ a measurable function. Concerning $(P)$, we have the following  result.
\begin{lem}\label{sol of (P)}
Let $g\in L^q(\Omega)$, $q>\frac{d}{p_-}$. Then for any $\lambda>0$, $(P)$ admits a unique weak solution $u\in \X$ satisfying
$$\int_\Om u\varphi\,dx+\la\int_\Om |\nabla u|^{p(x)-2}\nabla u.\nabla \varphi\,dx=\int_\Om g\varphi\,dx, \quad \forall
\varphi \in \X.$$
Furthermore, $u\in L^\infty(\Omega)$.
\end{lem}
\begin{rem} Lemma \ref{sol of (P)} still holds under the assumptions $p_->d$ and $q>1$. 
\end{rem}
\begin{proof}
Consider the energy functional $J_\la$ associated to $(P)$ given by
\[J_\la(u)=\frac{1}{2}\int_\Om u^2\,dx+\la\int_\Om\frac{|\nabla u|^{p(x)}}{p(x)}\,dx-\int_\Om gu\,dx.\]
Note that $J_\la$ is well-defined and G\^ateaux differentiable on $\X$. Indeed, $q>\frac{d}{p_-}$ and $1<p_-\leq p_+<d$ imply that $L^q\subset (L^{p^*(x)})'$.
By Theorem \ref{SE1} and \eqref{a}, for $\|u\|_{\X}\geq 1$:
\begin{align*}
 J_\la(u)&=\frac{1}{2}\int_\Om u^2\,dx+\la\int_\Om\frac{|\nabla u|^{p(x)}}{p(x)}\,dx-\int_\Om gu\,dx
\geq\frac{\lambda}{p_+}\|u\|_{\X}^{p_{-}}-C\|u\|_{\X}.
\end{align*}
Thus $J_\la$ is coercive. Furthermore $J_\lambda$ is continuous and strictly convex
on $\X$ and therefore admits a global minimizer $u\in \X$ which is a weak solution to $\eqref{(P)}$. In addition, applying Corollary \ref{reg2} in Appendix \ref{ApC}, $u\in L^\infty(\Omega)$.
\end{proof}

\noindent\textbf{Proof of Theorem  \ref{sol for (S_t)}:} Let $N\in\mathbb N^*$, $T>0$ and set $\Delta_t=\frac{T}{N}$.
For $0\leq n\leq N$, we define $t_n=n\Delta_t$. We perform the proof along five steps.\\
\textbf{Step 1.} Approximation of $h$.\\
For $n\in\{1,\cdots,N\}$, we define for $t\in[t_{n-1}, t_n)$ and $x \in \Omega$
$$h_{\Delta_t}(t,x)=h^n(x)\eqdef \frac{1}{\Delta_t}\int_{t_{n-1}}^{t_n} h(s,x)ds.$$
Then by Jensen's Inequality:
\begin{equation*}
\begin{split}
\|h_{\Delta_{t}}\|_{L^q(Q_T)}^q&=\Delta_t\DD\sum_{n=1}^N\|h^n\|^q_{L^{q}}=\Delta_t\DD\sum_{n=1}^N\|\frac{1}{\Delta_t}\int_{t_{n-1}}^{t_n}h(s,x)ds\|^q_{L^q}\\
&\leq\DD\sum_{n=1}^N\int_{t_{n-1}}^{t_n}\|h(s,.)\|_{L^q}^qds \leq \|h\|_{L^q(Q_T)}^q.
\end{split}
\end{equation*}
Thus $h_{\Delta_t}\in L^q(Q_T)$ and $h^n\in L^q(\Omega)$. Also note that $h_{\TT}\ra h$ in  $L^q(Q_T)$.
Indeed let $\varepsilon>0$, there exists $h^\var\in C^1_0(Q_T)$ such that $$\|h-h^\var\|_{L^q(Q_T)}<\frac{\var}{3}.$$
Hence
$$\|h_{\TT}-(h^\var)_{\TT}\|_{L^q(Q_T)}\leq\|h-h^\var\|_{L^q(Q_T)}<\frac{\var}{3}.$$
Since $\|h^\varepsilon-(h^\varepsilon)_{\TT}\|_{L^q(Q_T)}\ra 0$ as $\TT\ra0$, we have for $\Delta_t$ small enough
\[ \|h_{\TT}-h\|_{L^q(Q_T)}\leq
\|h_{\TT}-h^\varepsilon_{\TT}\|_{L^q(Q_T)}+\|h^\varepsilon-h^\varepsilon_{\TT}\|_{L^q(Q_T)}+
\|h-h^\varepsilon\|_{L^q(Q_T)}<\var.\]
Hence $h_{\Delta_t}\ra h$ in $L^q(Q_T)$.\\
\textbf{Step 2.} Time-discretization of $(S_T)$.\\
We define the following iterative scheme $u^0=u_0$ and for $n\geq 1$,
\begin{equation}\label{eq1}
\mbox{$u^n$ is solution of}\
\left\{\begin{array}{ll}
\DD\frac{u^{n}-u^{n-1}}{\Delta_{t}}-\Delta _{p(x)}u^n = h^n~&\text{ in }\Om,  \\
u^n = 0 ~&\text{  on   }~\pa\Om.
\end{array} \right.
\end{equation}
Note that the sequence $(u^n)_{n\in\{1,\cdots,N\}}$ is well-defined.
Indeed, existence and uniqueness of $u^1\in\X\cap L^\infty(\Om)$ follows from Lemma \ref{sol of (P)} with $g=\Delta_t h^1+u^0\in
L^q(\Omega)$. Hence by induction we
obtain in the same way the existence of $(u^n)$, for any $n=2,\cdots,N$.\\
Defining the functions, for $n=1,\cdots,N$ and $t\in[t_{n-1},t_n)$:
\begin{equation}\label{uu}
u_{\Delta_t}(t)=u^n \quad \mbox{and}\quad  \tilde u_{\Delta_t}(t)=\DD\frac{(t-t_{n-1})}{\Delta_t}(u^n-u^{n-1})+u^{n-1},
\end{equation}
we get
\begin{equation}\label{eq2}
 \frac{\pa \tilde u_{\TT}}{\pa t} -\Delta_{p(x)}u_{\TT}=h_{\TT} \quad\mbox{in } Q_T.
\end{equation}
\textbf{Step 3.} A priori estimates for $u_{\TT}$ and $\tilde u_{\TT}$.\\
Multiplying the equation in \eqref{eq1} by $(u^n-u^{n-1})$ and summing from $n=1$ to $N'\leq N$, we get
\begin{equation}\label{eq3}
\begin{split}
&\sum_{n=1}^{N'}\int_\Om|\nabla u^n|^{p(x)-2}\nabla u^n.\nabla(u^n-u^{n-1})dx  \\
&+\sum_{n=1}^{N'}\TT\int_\Om\left(\frac{u^n-u^{n-1}}{\TT}\right)^2dx=\DD\sum_{n=1}^{N'}\int_\Om h^n(u^n-u^{n-1})dx,
\end{split}
\end{equation}
hence by Young inequality and using the convexity of $u\ra\displaystyle\int_\Om\frac{|\nabla u|^{p(x)}}{p(x)}dx$ we obtain:
\begin{equation*}
\begin{split}
\sum_{n=1}^{N'}\int_\Om \frac{1}{p(x)}&\left(|\nabla u^n|^{p(x)}-|\nabla u^{n-1}|^{p(x)}\right)dx \\
&+\frac{1}{2}\sum_{n=1}^{N'}\TT\int_\Om\left(\frac{u^n-u^{n-1}}{\TT}\right)^2dx
\leq \frac{1}{2}\|h\|^2_{L^2(Q_T)}.
\end{split}
\end{equation*}
Thus we obtain
\begin{equation}\label{b1}
\DD\left(\frac{\pa\tilde u_{\Delta t}}{\pa t}\right)_{\Delta t} \mbox{ is bounded in } L^2(Q_T)  \mbox{ uniformly in  }
\TT,
\end{equation}
\begin{equation}\label{b2}
\mbox{$(u_{\TT}) \mbox{ and }(\tilde u_{\TT})$ are bounded in $L^\infty(0,T;\X)$ uniformly in
$\TT$.}
\end{equation}
Furthermore, using \eqref{b1} we have
\begin{equation}\label{eq7}
\sup_{[0,T]}\|u_{\TT}-\tilde u_{\TT}\|_{L^2(\Om)}
\leq \DD\max_{n=1,...,N}\|u^n-u^{n-1}\|_{L^2(\Om)}
\leq C\TT^{1/2}.
\end{equation}
Therefore for $\TT\ra 0^+$, there exist $u,v\in L^\infty(0,T,\X)$ such that (up to a subsequence)
\begin{equation}\label{fe}
\tilde u_{\TT}\stackrel{*}{\rightharpoonup}  u \text{ in }L^\infty(0,T,\X),\quad  u_{\TT}\stackrel{*}{\rightharpoonup} v
\text{ in }L^\infty(0,T,\X),
\end{equation}
\begin{equation}\label{convL2}
\frac{\pa\tilde u_{\TT}}{\pa t}\rightharpoonup \frac{\pa u}{\pa t}\text { in }L^2(Q_T).
\end{equation}
It follows from \eqref{eq7} that $u\equiv v$. By \eqref{fe}, for any $r\geq1$
\begin{equation}\label{convergeLp}
\tilde u_{\TT},\ u_{\TT} \rightharpoonup  u \text{ in }L^r(0,T;\X).
\end{equation}
\textbf{Step 4.} $u$ satisfies $(S_T)$.\\
\noindent Plugging \eqref{b1}, \eqref{b2} and since the embedding $\WW\hookrightarrow L^2(\Om)$ is compact, the Aubin-Simon's compactness result
(see
\cite{AS}) implies that (up to a subsequence),
\begin{equation}\label{new 1}
\tilde u_{\TT}\ra u\in C([0,T];L^2(\Om)). 
\end{equation}
Now multiplying (\ref{eq2}) by $(u_{\TT}-u)$ we get
\begin{equation*}\label{eq8}
\ii\frac{\pa\tilde u_{\TT}}{\pa t}(u_{\TT}-u)dxdt-\int_0^T\langle\Delta_{p(x)}u_{\TT},u_{\TT}-u\rangle dt
 =\int^T_0\int_\Om h_{\TT}(u_{\TT}-u)dxdt.
\end{equation*}
Rearranging the terms in the last equations and using \eqref{eq7}-\eqref{convergeLp} we have
\begin{equation*}
\begin{split}
\ii&\left(\frac{\pa\tilde u_{\TT}}{\pa t}-\frac{\pa u}{\pa t}\right)(\tilde u_{\TT}-u)dxdt
-\int_0^T\langle\Delta_{p(x)}u_{\TT}-\Delta_{p(x)}u,u_{\TT}-u\rangle dt\\
&=o_{\TT}(1)
\end{split}
\end{equation*}
where $o_{\TT}(1)\to 0$ as $\TT\to 0^+$. Thus we get
\begin{equation*}\label{eq10}
 \frac{1}{2}\int_\Om|\tilde u_{\TT}(T)-u(T)|^2
 dx-\int_0^T\langle\Delta_{p(x)}u_{\TT}-\Delta_{p(x)}u,u_{\TT}-u\rangle=o_{\TT}(1).
\end{equation*}
Using \eqref{new 1}, we obtain
\[\int_0^T\langle\Delta_{p(x)}u_{\TT}-\Delta_{p(x)}u,u_{\TT}-u\rangle dt=o_{\TT}(1)\]
and by Lemma \ref{A4} we conclude that
\begin{equation}\label{conve}
\DD\int_0^T\!\!\int_\Om|\nabla (u_{\TT}-u)|^{p(x)}dxdt\rightarrow 0.
\end{equation}
This implies $\nabla u_{\TT}$ converges to $\nabla u$ in $L^{p(x)}(Q_T)$ and $u_{\TT}$ converges to $u$ in $\X$.
Furthermore
\begin{equation}\label{cvc}
|\nabla u_{\TT}|^{p(x)-2}\nabla u_{\TT}\rightarrow |\nabla u|^{p(x)-2}\nabla u \mbox{ in } (L^{p_c(x)}(Q_T))^d.
\end{equation}
Indeed, we write
\begin{equation}\label{srhs}
\begin{split}
\int_{Q_T} ||\nabla u_{\TT}|^{p(x)-2}\nabla u_{\TT}&-|\nabla u|^{p(x)-2}\nabla
u|^{\frac{p(x)}{p(x)-1}}dxdt\\
&=\int_0^T\!\!\!\int_{\Omega\backslash \Omega_2}(...) dxdt + \int_0^T\!\!\!\int_{\Omega_2}(...) dxdt
\end{split}
\end{equation}
where $\Omega_2=\{x\in\Omega\,|\, p(x)>2\}$. We apply inequality \eqref{A0}. Then the first term in the right-hand side converges to zero as $\Delta_t\ra0^+$. For the second term, we apply H\"older inequality \eqref{hi}: 
\begin{equation*}
\begin{split}
\int_0^T\!\!\!\int_{\Omega_2}(...) dxdt &\leq \int_0^T\!\!\!\int_{\Omega_2} |\nabla(u_{\TT}-u)|^{\frac{p(x)}{p(x)-1}}(|\nabla
u_{\TT}|+|\nabla u|)^{\frac{p(x)(p(x)-2)}{p(x)-1}}dxdt \\
&\leq cXY
\end{split}
\end{equation*}
where 
$$X=\|\,|\nabla(u_{\TT}-u)|^{\frac{p(x)}{p(x)-1}}\|_{L^{p(x)-1}(Q_{T,2})},$$
$$Y= |\,(|\nabla
u_{\TT}|+|\nabla u|)^{\frac{p(x)(p(x)-2)}{p(x)-1}}\|_{L^{\frac{p(x)-1}{p(x)-2}}(Q_{T,2})},$$
$Q_{T,2}= (0,T)\times \Omega_2$ and $c\geq 1$ is a constant independent of $\TT$. We define $r=p_{|\Omega_2}$ the restriction of $p$ on $\Omega_2$. Then $r_-=2$ and $r_+=p_+$. With the new notations and applying Lemma \ref{A6}, we have
\begin{equation}\label{equ1}
X\leq \|\nabla(u_{\TT}-u)\|_{L^{p(x)}(Q_{T})}^{\frac{2}{p_+-1}}+ \|\nabla(u_{\TT}-u)\|_{L^{p(x)}(Q_{T})}^{p_+}
\end{equation} 
and
\begin{equation}\label{equ2}
\begin{split}
Y &\leq 1+\|\,|\nabla u_{\TT}|+|\nabla u|\,\|_{L^{r(x)}(Q_{T,2})}^{p_+(p_+-2)}\\
&\leq 1+\|\,|\nabla u_{\TT}|+|\nabla u|\,\|_{L^{p(x)}(Q_{T})}^{p_+(p_+-2)}\\
&\leq 1 + c(p)(\|\,\nabla u_{\TT}\|_{L^{p(x)}(Q_{T})}^{p_+(p_+-2)}+\| \nabla u\|_{L^{(x)}(Q_{T})}^{p_+(p_+-2)})\\
&\leq C(p)
\end{split}
\end{equation}
since $(u_{\TT})$ is bounded in $L^\infty(0,T;\X)$ uniformly in $\TT$.\\
Plugging \eqref{conve}, \eqref{equ1} and \eqref{equ2}, we deduce that the second term in the right-hand side of \eqref{srhs} converges to $0$ as $\TT\ra 0^+$. Hence we have \eqref{cvc}.\\
Finally, gathering Step 1., \eqref{convL2} and \eqref{cvc}, we conclude passing to the limit, in the distribution sense, in equation \eqref{eq2} that $u$ is a weak solution of $(S_T)$. Furthermore $u$ is the unique weak solution of $(S_T)$. Indeed assume that there exists $v$ a weak solution of $(S_T)$.
Thus,
$$\ii \frac{\partial(u-v)}{\partial t}(u-v)\,dxdt-\int_0^T
\langle\Delta_{p(x)}u-\Delta_{p(x)}v,u-v\rangle\,dt=0.$$
Since $u(0)=v(0)$, the above equality implies that $u\equiv v$.\\
\textbf{Step 5.} $u$ belongs to $C([0,T];\X)$.\\
Since $u\in C([0,T]; L^2(\Om))\cap L^\infty([0,T];\X)$ and $p\in \mathcal P^{\log}(\Omega)$, $u:t\in[0,T]\ra\X$ is weakly continuous.\\ 
Fix $t_0\in [0,T]$. Since $\rho_p$ is weakly lower semicontinuous (see Theorem 3.2.9 in \cite{DHHR}) we have
\[\int_\Om\frac{|\nabla u(t_0)|^{p(x)}}{p(x)}dx\leq\liminf_{t\ra t_0}\int_\Om\frac{|\nabla u(t)|^{p(x)}}{p(x)}dx.\]
From \eqref{eq3} with $\sum_{n=N''}^{N'}$ for $1\leq N''\leq N'$ and since $|\nabla u_{\TT}|$ converges to $|\nabla u|$ in
$L^{p(x)}$, it follows that $u$ satisfies for any
$t\in [t_0,T]$:
\begin{equation}\label{eq10b}
\begin{split}
\int_{t_0}^t\int_\Om \left(\frac{\partial u}{\partial t}\right)^2\,dxds+\int_\Om \frac{|\nabla
u(t)|^{p(x)}}{p(x)}\,dx-\int_\Om &\frac{|\nabla u(t_0)|^{p(x)}}{p(x)}\,dx\\
&\leq \int_{t_0}^t\int_\Om h\frac{\partial u}{\partial t}\,dxds.
\end{split}
\end{equation}
Passing to the limit, we get
\[\limsup_{t\ra t^+_0}\int_\Om\frac{|\nabla u(t)|^{p(x)}}{p(x)}dx\leq\int_\Om\frac{|\nabla u(t_0)|^{p(x)}}{p(x)}dx.\]
Define $ v(t)=\nabla u(t)/(p(x))^{1/p(x)}.$ Thus we get $\DD\lim_{t\ra t^+_0}\rho_p(v(t))=\rho_p(v(t_0))$.\\ 
Now we prove the left continuity. Let $0<k\leq t-t_0$. Multiplying $(S_T)$ by $\tau_k(u)(s)=\frac{u(s+k)-u(s)}{k}$ and integrating over $(t_0,t)\times \Omega$, the convexity gives
\begin{equation}\label{eq10t}
\begin{split}
\int_{t_0}^t\int_\Om \tau_k(u)\frac{\partial u}{\partial t}\,dxds&+\int_{t}^{t+k}\int_\Om \frac{|\nabla
u(s)|^{p(x)}}{kp(x)}\,dxds\\
&-\int_{t_0}^{t_0+k}\int_\Om \frac{|\nabla u(s)|^{p(x)}}{kp(x)}\,dxds\\
&\geq \int_{t_0}^t\int_\Om \tau_k(u)h\,dxdt.
\end{split}
\end{equation}
By Dominated Convergence Theorem as $k\rightarrow 0^+$:
\begin{equation*}
\begin{split}
&\int_{t}^{t+k}\int_\Om \frac{|\nabla u(s)|^{p(x)}}{kp(x)}\,dxds\rightarrow \int_\Om \frac{|\nabla
u(t)|^{p(x)}}{p(x)}\,dx,\\
&\int_{t_0}^{t_0+k}\int_\Om \frac{|\nabla u(s)|^{p(x)}}{kp(x)}\,dxds\rightarrow\int_\Om \frac{|\nabla
u(t_0)|^{p(x)}}{p(x)}\,dx.
\end{split}
\end{equation*}
Hence \eqref{eq10t} yields
\begin{equation*}\label{eq10q}
\begin{split}
\int_{t_0}^t\int_\Om \left(\frac{\partial u}{\partial t}\right)^2\,dxds+\int_\Om \frac{|\nabla
u(t)|^{p(x)}}{p(x)}\,dx-\int_\Om &\frac{|\nabla u(t_0)|^{p(x)}}{p(x)}\,dx\\
&\geq\int_{t_0}^t\int_\Om h\frac{\partial u}{\partial t}\,dxds.
\end{split}
\end{equation*}
From the above inequality, we deduce that we have the equality in \eqref{eq10b}. This implies, using the  Dominated Convergence Theorem, that
$\rho_p(v(t))\rightarrow \rho_p(v(t_0))$ as $t\rightarrow t_0$.\\
Since $v(t)\rightharpoonup v(t_0)$ in $L^{p(x)}(\Om)$ and $\rho_p(v(t))\rightarrow \rho_p(v(t_0))$ as $t\ra t_0$, Lemma \ref{A5} implies the convergence of $v(t)$ to $v(t_0)$ in $L^{p(x)}(\Om)$. Therefore we deduce that $u\in C([0,T];\X)$.\ $\hfill \square$
\section{Existence of solution of $(P_T)$}
\noindent \textbf{Proof of Theorem \ref{PT1}:} We proceed as in the proof of Theorem \ref{sol for (S_t)} splitting the proof in several steps.\\ 
{\bf Step 1.} Existence of barrier functions.\\
Consider the equations, for $i\in\{0,1,2\}$
\begin{equation}\label{CL}
 \left\{\begin{array}{ll}
 \DD\frac{dv_i}{dt} = L_i(v_i),& \\
 v_i(0)  =(-1)^i\kappa,&
 \end{array}\right.
\end{equation}
where $\kappa= \|u_0\|_{\infty}$.\\
By Cauchy-Lipschitz Theorem, there exists $T_i^{max}\in(0,+\infty]$ and a unique maximal solution $v_i$ to \eqref{CL} on $[0,T_i^{max})$.\\
If {\bf(H1)} holds, we take $T\in (0,T_0^{max})$ otherwise, if {\bf (H2)} holds, we take $T\in (0,\min(T_1^{max};T_2^{max}))$.\\
Let $N\in \N^*$. Set $\Delta_t=\frac{T}{N}$ and consider the family $(v_i^n)$ defined by
$v_i^n=v_i(t_n)=v_i(n\Delta_t)$ for $n\in \{1,...,N\}$. Hence for any $i\in\{0,1,2\}$
\begin{equation*}\label{eq 12}
v_i^{n+1}=v_i^{n}+\int_{t_n}^{t_{n+1}}  L_i(v_i(s))ds,\quad \forall n\in \{0,...,N-1\}.
\end{equation*}
Replacing $L_1$ (resp. $L_2$) by $\min(L_1,0)$ (resp. $\max(L_2,0)$) in {\bf (H2)}, we can assume that $L_1\leq0$ and $L_2\geq0$. We get for $n\in \{0,...,N\}$, $v_1(T) \leq v_1^{n}\leq -\kappa$ and for $i=0$ or $i=2$, $\kappa \leq v_i^{n}\leq v_i(T)$.\\
\textbf{Step 2.} Semi-discretization in time of $(P_T)$.\\
Introduce the following iterative scheme $(u^n)$ defined as
\begin{equation*}\label{eq 13}
u^0=u_0\quad \mbox{ and }\quad
\Biggl\{\begin{array}{ll}
u^{n}-\Delta_t\Delta _{p(x)}u^n = u^{n-1}+\Delta_t f(x, u^{n-1})~&\text{ in }\Om,  \\
u^n = 0 ~&\text{  on   }~\pa\Om.
\end{array}
\end{equation*}
We just prove the existence of $u^1$. The conditions $\bf{(f_1)}$-$\bf{(f_2)}$ insures that $f(.,u^0)\in L^q(\Omega)$ with $q>\frac{d}{p_-}$. Thus Lemma \ref{sol of (P)} applying with $g=u^0+\Delta_tf(x,u^0)\in L^q(\Omega)$ gives the existence of $u^1\in
\X\cap L^\infty(\Om)$.\\ 
Let $u_{\TT}$ and $\tilde u_{\TT}$ be defined as in \eqref{uu} and for $t<0$, $u_{\TT}(t)=u_0$. Thus \eqref{eq2} is satisfied with $h_{\TT}(t)\eqdef f(x,u_{\TT}(t-\TT))$.\\
\textbf{Step 3.} $(u^n)$ is bounded in $L^\infty(\Om)$ uniformly in $\Delta_t$.\\
First we consider the case where {\bf (H1)} is valid. We aim that for all $n$, $|u^n|\leq v_0^n$ in $\Omega$. We just prove for $n=1$. Since $L_0$ and $v_0$ are nondecreasing, we get
\begin{equation*}
u^1-v_0^1- \Delta_t\Delta_{p(x)}u^1=\int_0^{\Delta_t} f(x,u_0)-L_0(v_0(s))ds+u_0-v_0^0\leq 0.
\end{equation*}
Multiplying the previous inequality by $(u^1-v_0^1)^+=\max(u^1-v_0^1,0)$ and integrating on $\mathcal O=\{x\in \Omega \ |\ u^1(x)>v_0^1\}$, we get
$$\int_{\mathcal O} (u^1-v_0^1)^2dx +\Delta_t\int_{\mathcal O} |\nabla u^1|^{p(x)}dx\leq 0.$$
Hence, $u^1\leq v_0^1$ and by the same method we have $-v_0^1\leq u^1$.\\
For {\bf (H2)} we claim that for all $n$, $v_1^n\leq u^n \leq v_2^n \mbox{ in } \Om.$ Let $n=1$. Since $L_1$, $L_2$, $-v_1$  and $v_2$ are nondecreasing:
\begin{equation*}
\begin{split}
&u^1-v_1^1- \Delta_t\Delta_{p(x)}u^1=\int_0^{\Delta_t} f(x,u_0)-L_1(v_1(s))ds+u_0-v_1^0\geq 0,\\
&u^1-v_2^1- \Delta_t\Delta_{p(x)}u^1=\int_0^{\Delta_t} f(x,u_0)-L_2(v_2(s))ds+u_0-v_2^0\leq 0.
\end{split}
\end{equation*}
Multiply the first inequality by $(v^1_1-u^1)^+$ and the second inequality by $(u^1-v^1_2)^+$. Integrating respectively
on
$\mathcal O_1=\{x\in \Omega \ |\  v_1^1>u^1(x)\}$ and $\mathcal O_2=\{x\in \Omega \ |\  v_2^1<u^1(x)\}$, we get
\begin{equation*}
\begin{split}
-&\int_{\mathcal O_1} (u^1-v_1^1)^2dx -\Delta_t\int_{\mathcal O_1} |\nabla u^1|^{p(x)}dx\geq 0,\\
&\int_{\mathcal O_2} (u^1-v_2^1)^2dx +\Delta_t\int_{\mathcal O_2} |\nabla u^1|^{p(x)}dx\leq 0.
\end{split}
\end{equation*}
Then $v_1^1 \leq u^1\leq v_2^1 $. By induction, we deduce that for $n\in \{0,...,N\}$, $v^n_1\leq u^n \leq v_2^n$ in $\Om$.\\
Thus we have 
\begin{equation}\label{eq20}
(u_{\TT}),\ (\tilde u_{\TT})  \mbox{ are bounded in $L^\infty(Q_T)$ uniformly in $\TT$.}
\end{equation}
and
\begin{equation*}\label{eq21}
(h_{\TT}) \mbox{ is bounded in $L^2(Q_T)$ uniformly in $\TT$}
\end{equation*}
Indeed, either {\bf (H1)} holds which implies $$|f(x,u^{n})|\leq L_0(u^n)\leq L_0(v_0(T))$$ or {\bf (H2)} holds, we have $$|f(x,u^{n})|\leq \max(-L_1(u^n),L_2(u^n))\leq \max(-L_1(v_1(T)),L_2(v_2(T))).$$ 
Hence
\begin{equation*}
\|h_{\Delta_{t}}\|_{L^2(Q_T)}^2=\Delta_t\DD\sum_{n=1}^N\|f(x,u^{n-1})\|^2_{L^{2}(\Om)}\leq C.
\end{equation*}
{\bf Step 4.} End of the proof.\\
By the same computations of Step 3. of the proof of Theorem \ref{sol for (S_t)}, we obtain estimates and we prove there exists $u\in L^\infty(0,T,\X)$ such
that
\begin{equation*}
\tilde u_{\TT},\, u_{\TT}\stackrel{*}{\rightharpoonup}  u  \text{ in }L^\infty(0,T,\X)\
\mbox{ and }\ \frac{\pa\tilde u_{\TT}}{\pa t}\rightharpoonup \frac{\pa u}{\pa t}\text { in }L^2(Q_T).
\end{equation*}
\eqref{b1} implies that $(\tilde u_{\TT})$ is equicontinuous in $C([0,T];L^r(\Om))$ for $1\leq r\leq 2$. By the
interpolation inequality  and \eqref{eq20} we obtain $(\tilde u_{\TT})$ is equicontinuous in $C([0,T];L^r(\Om))$ for
any $r>1$.\\ By \eqref{b2} and Theorem \ref{SE1}, we deduce applying the Ascoli-Arzela Theorem that (up to a subsequence)
for any $r>1$
$$\quad \tilde u_{\TT}\ra u \mbox{ in } C([0,T];L^r(\Om)).$$
Since $(u_{\TT})$ is uniformly bounded in $L^\infty(Q_T)$, $\bf{(f_1)}$ implies
\begin{equation*}
\|h_{\TT}(t)-f(.,u(t))\|_{L^2(\Om)}\leq C\|u_{\TT}(t-\TT)-u(t)\|_{L^2(\Om)}.
\end{equation*}
Hence we deduce that $h_{\TT}\ra f(.,u)$ in $L^\infty(0,T;L^2(\Om))$. Next we follow Step 4 of Theorem \ref{sol for (S_t)} and obtain that $u$ is a weak solution to $(P_T)$.\\
Now, we prove the uniqueness of the solution to $(P_T)$. Let $w$ another weak solution of $(P_T)$. By $\bf{(f_1)}$, for $t\in [0,T]$:
\begin{equation*}
\begin{split}
\frac12& \|u(t)-w(t)\|_{L^2(\Omega)}^2-\int_0^t\langle \Delta_{p(x)}u-\Delta_{p(x)}w,u-w\rangle ds\\
&=\ii
(f(x,u)-f(x,w))(u-w)dxds \leq C\int_0^t \|u(s)-w(s)\|_{L^2(\Omega)}^2 ds.
\end{split}
\end{equation*}
Since  $u\ra \Delta_{p(x)}u$ is a monotone operator from $\X$ to $\X'$, the second term in the left-hand side is nonnegative. Then, by Gronwall's Lemma, we deduce that $u\equiv w$.\\
Step 5 of the proof of Theorem \ref{sol for (S_t)} again go through and completes the proof.
$$\eqno \square$$

\noi Now we give the proof of Theorem \ref{PT3}.\ \\
\noindent \textbf{Proof of Theorem \ref{PT3}:} Firt we introduce the stationary quasilinear elliptic problem asociated to $(P_T)$:
\begin{equation*}\label{E}
 \left\{
\begin{array}{l l}
-\Delta_{p(x)} u=f(x,u)& \mbox{ in } \Omega,\\
u=0& \mbox{ on } \partial \Om.
\end{array}\right.\tag{$E$}
\end{equation*}
Thus we claim that if {\bf (C1)} or {\bf (C2)} holds there exist $\underline u,\overline u\in \X\cap L^\infty(\Om)$, a sub- and a supersolution of \eqref{E} such that $\underline u\leq u_0\leq\overline u$.\\
First, consider that {\bf (C1)} holds. For $(x,s)\in \Omega\times\R$, define 
$$G(x,s)=|\Delta_{p(x)} u_0(x)|+|f(x,s)|.$$
Consider the following problems:
\begin{equation*}
\left\{
\begin{array}{l l}
-\Delta_{p(x)}\underline u=-G(x,\underline u)& \mbox{ in } \Omega,\\
u=0& \mbox{ on } \partial \Om
\end{array}\right.
\end{equation*}
and
\begin{equation*}
 \left\{
\begin{array}{l l}
-\Delta_{p(x)}\overline u=G(x,\overline u)& \mbox{ in } \Omega,\\
u=0& \mbox{ on } \partial \Om.
\end{array}\right.
\end{equation*}
The existence of $\underline u$ and $\overline u\in \X$ follows from the sub-homogeneity of $f$ given by $\bf{(f_3)}$ (see Theorem 4.3 in \cite{FQ}) and by Corollary \ref{reg2} we have $\underline u,\ \overline u\in L^\infty(\Om)$.\\
Moreover
$$-\Delta_{p(x)}\underline u=-G(.,\underline u)\leq -\Delta_{p(x)}u_0 \ \mbox{ a.e in }\ \Omega.$$
and
$$  -\Delta_{p(x)}\overline u=G(.,\overline u)\geq -\Delta_{p(x)}u_0 \ \mbox{ a.e in }\ \Omega.$$
Hence Lemma \ref{wc} implies $\underline u\leq u_0$ and $\underline u$ is a subsolution of $(E)$.
Similarly we have that $\overline u\geq u_0$ and $\overline u$ is a supersolution of $(E)$.\\
Now, if {\bf (C2)} holds. We have the following lemma which follows from \cite{Fa,YY}:
\begin{lem}\label{31YY}
Let $p\in C^1(\overline{\Om})$ and $\lambda\in \R^+$. Let $w_\lambda\in \X\cap L^\infty(\Om)$ be the unique solution of
\begin{equation}\label{El}
 \left\{
\begin{array}{l l}
-\Delta_{p(x)} w_\lambda=\lambda& \mbox{ in } \Omega,\\
w_\lambda=0& \mbox{ on } \partial \Om.
\end{array}\right.
\end{equation}
Then, there exists two constants $C_1$ and $C_2$ which do not depend to $\lambda$ such that 
$$ \|w_\lambda\|_{L^{\infty}}\leq C_1\lambda^{\frac{1}{p_--1}} \quad \mbox{and} \quad  
w_\lambda(x) \geq C_2 \lambda^{\frac{1}{p_+-1+\mu}}\mbox{dist}(x,\partial \Omega)$$
where $\mu\in (0,1)$.
\end{lem}
Fix $\lambda>0$, let $w_\lambda$ the solution of \eqref{El}. Since $\beta < p_--1$ and by Lemma \ref{31YY}: for $\lambda$ large enough, $w_\lambda$ verifies
\begin{equation}\label{sur}
-\Delta_{p(x)} w_\lambda=\lambda \geq C\left(1+ C_1^\beta\lambda^{\frac{\beta}{p_--1}}\right)\geq C(1+w_\lambda^\beta)\geq |f(x,w_\lambda)|.  
\end{equation}
Moreover, since $u_0\in C^1_0(\overline \Omega)$, there exists $K>0$ such that for any $x\in \Omega$, $|u_0(x)|\leq K\mbox{dist}(x,\partial \Omega)$. Hence choosing $\lambda$ large enough, we have by Lemma \ref{31YY} $w_\lambda\geq |u_0|$ in $\overline \Omega$.\\
Set $\overline u= w_\lambda$ and $\underline u=- w_\lambda$. We deduce for $\lambda$ large enough$,  \overline u$ and $\underline u$ are a super- and a subsolution of \eqref{E} such that $\underline u\leq u_0\leq\overline u$.\\
Now we proceed as the proof of Theorem \ref{PT1}. We define the sequence $(u^n)$ as follows.
\begin{equation*}\label{iteration}
\left\{
\begin{array}{l l}
u^{n}-\Delta_t\Delta_{p(x)} u^{n}=u^{n-1}+\Delta_t f(x,u^{n-1})& \mbox{ in } \Omega,\\
u^n=0& \mbox{ on } \partial \Om
\end{array}\right.
\end{equation*}
for $n=1,2,...,N$ with $u^0=u_0$. we prove for $n\geq 1$, $\underline u\leq u^n\leq \overline u$ in $\Om$.
Indeed for $n=1$, we have
\[\underline u-u^1-\Delta_t(\Delta_{p(x)}\underline u-\Delta_{p(x)}u^1)\leq \underline
u-u^0+\Delta_t(f(x,u^0)-f(x,\underline u)).\]
Let $\La$ be the Lipschitz constant of $f$ on $[-M,M]$, where $M$ is the maximum of $\|\underline u\|_{L^\infty}$ and $\|\overline u\|_{L^\infty}$.
Then
\[\underline u-u^1-\Delta_t(\Delta_{p(x)}\underline u-\Delta_{p(x)}u^1)\leq (Id-\Delta_tf)(\underline u-u^0).\]
For $\Delta_t$ small enough, the function $Id-\Delta_t f$ is nondecreasing. Then the right-hand side of the above inequality is nonpositive and thus by Lemma \ref{wc}
we have $\underline u\leq u^1$. Similarly we prove $u^1\leq \overline u$.

By induction, for $n\geq 1$, $\underline u\leq u^n\leq\overline u$ in $\Om$. Thus $(u^n)$ is uniformly
bounded in $L^\infty(\Om)$.
The rest of the proof follows Step 3 and 4 of the proof of Theorem \ref{PT1}.\hfill$\square$

\section{Existence of mild solutions and stabilization}
In this section we prove Theorems \ref{continuity de u}, \ref{continuity u with f} and \ref{stabilization}. We first show the m-accretivity of $A=-\Delta_{p(x)}$:
\begin{prop}\label{maccretive}
Let $f$ be locally Lipschitz and nonincreasing in respect to the second variable. Assume further that $f$ sastifies ${\bf (f_4)}$. Then, $A_f$ defined by $A_f(u)\eqdef-\Delta_{p(x)}u-f(.,u)$, is m-accretive in $L^\infty(\Omega)$.
\end{prop}
\begin{proof}
 First, let $h\in L^\infty(\Omega)$ and $\lambda>0$. Then,
\begin{equation*}
\left\{\begin{array}{ll}
 u+\lambda A_f(u)=h &\mbox{in }\Omega,\\
 u =0, &\mbox{on } \partial\Omega
\end{array}\right.
\end{equation*}
admits a unique solution, $u\in\X\cap L^\infty(\Omega)$. 
Indeed, for $\mu>0$ large enough $w_{\mu}$ and $-w_{\mu}$ defined in \eqref{El} in Lemma \ref{31YY} are respectively supersolution and subsolution to the above equation and then from the weak comparison principle, $u\in [-w_{\mu},w_\mu]$ and $u$ is obtained by a minimization argument and  a truncation argument. The uniqueness of the solution follows from the strict convexity of the associated energy functional. Next we prove the accretivity of $A_f$. Let $h$ and $g\in L^\infty(\Omega)$ and set $u$ and $v$ the unique solutions to
\begin{equation*}
\begin{aligned}
& u+\lambda A_fu=h\quad\mbox{in }\,\Omega,\\
& v+\lambda A_fv=g\quad\mbox{in }\,\Omega.
\end{aligned}
\end{equation*}
Substracting the two above equations and using the test function $w=(u-v-\|h-g\|_{L^\infty(\Omega)})^+$, we get $u-v\leq\|h-g\|_{L^\infty(\Omega)}$ and reversing the roles of $u$ and $v$, we get that $\|u-v\|_{L^\infty(\Omega)}\leq\|h-g\|_{L^\infty(\Omega)}$. This proves the proposition.
\end{proof}
\noindent Next, we prove Theorem \ref{continuity de u}.\\
\noindent{\bf Proof of Theorem \ref{continuity de u}:} We follow the approach in the proof of Theorems 4.2 and 4.4 in \cite{Ba}. Let $u_0$, $v_0$ be in $\overline{{\mathcal D}(A)}^{L^\infty(\Omega)}$.
For $z\in {\mathcal D}(A)$ and $r,k$ in $L^\infty(Q_T)$, set
$$
\varphi(t,s)=\|r(t)-k(s)\|_{L^\infty(\Om)}\quad \forall \, (t,s)\in [0,T]\times [0,T];
$$
\begin{equation*}
\begin{split}
b(t,r,k)&=\|u_0-z\|_{L^\infty(\Om)}+ \|v_0-z\|_{L^\infty(\Om)}+|t|\|A z \|_{L^\infty(\Om)}\\
&+\int_0^{t^+}\|r(\tau)\|_{L^\infty(\Om)} {\rm d}\tau+\int_0^{t^-}\|k(\tau)\|_{L^\infty(\Om)} {\rm d}\tau,\;t\in [-T,T],
\end{split}
\end{equation*}
and
$$
\Psi(t,s)=b(t-s,r,k)+
\left\{
\begin{array}{ll}
\displaystyle \int_0^{s}\varphi(t-s+\tau,\tau){\rm d}\tau & \mbox{ if } 0\leq s\leq t\leq T,\\
\displaystyle \int_0^{t}\varphi(\tau,s-t+\tau) {\rm d}\tau &  \mbox{ if } 0\leq t\leq s\leq T,
\end{array}
\right.
$$
the solution of
\begin{equation}\label{eqPsi}
\left\{
\begin{array}{rcll}
\displaystyle \frac{\partial \Psi}{\partial t}(t,s)+\frac{\partial \Psi}{\partial s}(t,s)&=&\displaystyle\varphi(t,s)\; &(t,s)\in [0,T]\times [0,T],\\
\displaystyle \Psi(t,0)&=& \displaystyle b(t,r,k)\; &t\in [0,T],\\
 \displaystyle  \Psi(0,s)& =& \displaystyle b(-s,r,k)\; &s\in [0,T].
\end{array}
\right.
\end{equation}
Moreover, let denote by $(u_\epsilon^n)$ the solution of \eqref{eq1} with $\Delta_t=\epsilon$, $h=r$, $r^n=\frac{1}{\epsilon}\int_{(n-1)\epsilon}^{n\epsilon}r(\tau,\cdot)d\tau$ and  $(u_\eta^n)$ the solution of \eqref{eq1} with  $\Delta_t=\eta$, $h=k$, $k^n=\frac{1}{\eta}\int_{(n-1)\eta}^{n\eta}k(\tau,\cdot)d\tau$ respectively. For $(n,m)\in \mathbb{N}^*$ elementary calculations lead to
\begin{equation*}
\begin{split}
u_\epsilon^n-u_\eta^m+\frac{\epsilon\eta}{\epsilon+\eta}(A u_\epsilon^n-A u_\eta^m)=&\frac{\eta}{\epsilon+\eta}(u_\epsilon^{n-1}-u_\eta^m)\\
+\frac{\epsilon}{\epsilon+\eta}(u_\epsilon^n-u_\eta^{m-1})+&\frac{\epsilon\eta}{\epsilon+\eta}(r^n-k^m),
\end{split}
\end{equation*}
and since $A$ is m-accretive in $L^\infty(\Omega)$ we first verify that $\Phi_{n,m}^{\epsilon,\eta}=\|u_\epsilon^n-u_\eta^m\|_{L^\infty(\Om)}$ obeys
\begin{equation*}
\begin{split}
&\Phi_{n,m}^{\epsilon,\eta}\leq \frac{\eta}{\epsilon+\eta}\Phi_{n-1,m}^{\epsilon,\eta}+\frac{\epsilon}{\epsilon+\eta}\Phi_{n,m-1}^{\epsilon,\eta}+\frac{\epsilon\eta}{\epsilon+\eta}\|r^n-k^m\|_\infty,\\
&\Phi_{n,0}^{\epsilon,\eta}\leq  b(t_n,r_\epsilon,k_\eta)\quad\mbox{ and }\quad  \Phi_{0,m}^{\epsilon,\eta}\leq  b(-s_m ,r_\epsilon,k_\eta),
\end{split}
\end{equation*}
and thus, with an easy inductive argument, that $\Phi_{n,m}^{\epsilon,\eta}\leq \Psi_{n,m}^{\epsilon,\eta}$ where $\Psi_{n,m}^{\epsilon,\eta}$ satisfies
\begin{equation*}
\begin{split}
&\Psi_{n,m}^{\epsilon,\eta}= \frac{\eta}{\epsilon+\eta}\Psi_{n-1,m}^{\epsilon,\eta}+\frac{\epsilon}{\epsilon+\eta}\Psi_{n,m-1}^{\epsilon,\eta}+\frac{\epsilon\eta}{\epsilon+\eta}\|h_\epsilon^n-h_\eta^m\|_\infty,\\
&\Psi_{n,0}^{\epsilon,\eta}=  b(t_n,r_\epsilon,k_\eta)\quad\mbox{ and }\quad  \Psi_{0,m}^{\epsilon,\eta}=  b(-s_m ,r_\epsilon,k_\eta).
\end{split}
\end{equation*}
For $(t,s)\in (t_{n-1},t_n)\times (s_{m-1},s_m)$, set
\begin{equation*}
\begin{split}
&\varphi^{\epsilon,\eta}(t,s)=\|r_\epsilon(t)-k_\eta(s)\|_{\infty},\\
&\Psi^{\epsilon,\eta}(t,s)=\Psi_{n,m}^{\epsilon,\eta},\\
&b_{\epsilon,\eta}(t,r,k)=b(t_n,r_\epsilon,k_\eta)
\end{split}
\end{equation*}
and $$b_{\epsilon,\eta}(-s,r,k)=b(-s_m,r_\epsilon,k_\eta).$$
 Then $\Psi^{\epsilon,\eta}$ satisfies the following discrete version of \eqref{eqPsi}:
\begin{equation*}
\begin{split}
&\frac{\Psi^{\epsilon,\eta}(t,s)-\Psi^{\epsilon,\eta}(t-\epsilon,s)}{\epsilon}+\frac{\Psi^{\epsilon,\eta}(t,s)-\Psi^{\epsilon,\eta}(t,s-\eta)}{\eta}= \varphi^{\epsilon,\eta}(t,s),\\
&\Psi^{\epsilon,\eta}(t,0)= b_{\epsilon,\eta}(t,r,k)\quad\mbox{ and }\quad  \Psi^{\epsilon,\eta}(0,s)=  b_{\epsilon,\eta}(s,r,k),
\end{split}
\end{equation*}
and from $b_{\epsilon,\eta}(\cdot,r,k)\to b(\cdot,r,k)$ in $L^\infty([0,T])$. Furthermore,
\begin{equation*}
\begin{split}
&\sum_{n=1}^{N_n}\int_{t_{n-1}}^{t_n}\Vert r(s)-r_n\Vert_\infty{\rm d}s\to 0, \quad\mbox{ as } \epsilon\to 0^+, \\
&\sum_{m=1}^{N_m}\int_{s_{m-1}}^{s_m}\Vert k(s)-k_m\Vert_\infty{\rm d}\tau\to 0\quad\mbox{ as }\eta\to 0^+.
\end{split}
\end{equation*}
The above statements follow easily from the fact that $r, k\in L^1(0,T; L^\infty(\Omega))$ and a density argument.\\
Then, we deduce that $\rho_{\epsilon,\eta}=\|\Psi^{\epsilon,\eta}-\Psi\|_{L^\infty([0,T]\times [0,T])} \to 0$ as $(\epsilon,\eta)\to 0$ (see for instance  \cite[Chap.4, Lemma~4.3, p.~136]{Ba} and \cite[Chap.4, proof of Theorem~4.1, p.~138]{Ba}). Then from
\begin{equation}
\|u_\epsilon(t)-u_\eta(s)\|_{\infty}=\Phi^{\epsilon, \eta}(t,s)\leq \Psi^{\epsilon, \eta}(t,s)\leq \Psi(t,s)+\rho_{\epsilon,\eta},\label{eqconv}
\end{equation}
we obtain with $t=s$, $r=k=h$, $v_0=u_0$:
$$
\|u_\epsilon(t)-u_\eta(t)\|_{L^\infty(\Om)}\leq 2\|u_0-z\|_{L^\infty(\Om)}+\rho_{\epsilon, \eta},
$$
and since $z$ can be chosen in $\mathcal{D}(A)$ arbitrary close to $u_0$, we deduce that $u_\epsilon$ is a Cauchy sequence in $L^\infty(Q_T)$ and then that $u_\epsilon \to u$ in $L^\infty(Q_T)$. Thus, passing to the limit in \eqref{eqconv} with $r=k=h$, $v_0=u_0$ we obtain 
\begin{equation*}
\begin{split}
\|u(t)-u(s)\|_{L^\infty(\Om)}&\leq \int_{0}^{\max(t,s)}\!\!\!\|h(|t-s|+\tau)-h(\tau)\|_{L^\infty(\Om)}{\rm d}\tau\\
&+2\|u_0-z\|_{L^\infty(\Om)}+\int_0^{|t-s|}\|h(\tau)\|_{L^\infty(\Om)}{\rm d}\tau\\
&+|t-s|\|A z \|_{L^\infty(\Om)},
\end{split}
\end{equation*}
which, together with the density ${\mathcal D}(A)$ in $L^\infty(\Omega)$ and  $h\in L^1(0,T;L^\infty(\Omega))$, yields $u\in C([0,T];L^\infty(\Omega))$.\\
Analogously, from \eqref{eqconv} with $\varepsilon=\eta=\Delta_t$, $r=k=h$, $v_0=u_0$ and $t=s+\Delta_t$ we deduce that
\begin{equation*}
\begin{split}
\|u_{\Delta_t}(t)-\Tilde u_{\Delta_t}(t)\|_{L^\infty(\Om)}&\leq 2\| u_{\Delta_t}(t)- u_{\Delta_t}(t-\Delta_t)\|_{L^\infty(\Om)}\\
& \leq 4\|u_0-z\|_{L^\infty(\Om)}+2\int_{0}^{t}\|h(\Delta_t+\tau)-h(\tau)\|_{\infty}{\rm d}\tau\\
&+2\int_0^{\Delta_t } \|h(\tau)\|_{L^\infty(\Om)}{\rm d}\tau+2\Delta_t\|A z \|_{L^\infty(\Om)}
\end{split}
\end{equation*}
which gives the limit $\Tilde u_{\Delta_t}\to u$ in $C([0,T];L^\infty(\Omega))$ as $\Delta_t\to 0^+$. Note that since $\Tilde u_{\Delta_t}\in C([0,T];C_0(\overline{\Omega}))$, the uniform limit $u$ belongs to $C([0,T];C_0(\overline{\Omega}))$. 
Moreover, passing to the limit in \eqref{eqconv} with $t=s$ we obtain
$$
\|u(t)-v(t)\|_{L^\infty(\Om)}\leq \|u_0-z\|_{L^\infty(\Om)}+\|v_0-z\|_{L^\infty(\Om)}+\int_0^t\|r(\tau)-k(\tau)\|_{\infty} {\rm d}\tau,
$$
and \eqref{estimation2} follows since we can choose $z$ arbitrary close to $v_0$. 
Finally, if $A u_0\in L^\infty(\Omega)$ and $h\in W^{1,1}(0,T;L^\infty(\Omega))$ and if we assume (without loss of generality) that $t>s$ then with $z=v_0=u(t-s)$ and $(r,k)=(h,h(\cdot+t-s))$ in the last above inequality we obtain 
\begin{equation}\label{mes}
\begin{split}
\|u(t)-u(s)\|_{L^\infty(\Om)}\leq  \|u_0&-u(t-s)\|_{L^\infty(\Om)}\\
&+\int_{0}^{s}\|h(\tau)-h(\tau+t-s)\|_{L^\infty(\Om)}{\rm d}\tau.
\end{split}
\end{equation}
From  \eqref{estimation2} with $v=u_0$, $k=A u_0$: 
\begin{equation}\label{mest}
 \|u_0-u(t-s)\|_{L^\infty(\Om)}\leq \int_0^{t-s}\|A u_0-h(\tau)\|_{L^\infty(\Om)}{\rm d}\tau.
\end{equation}
Using \eqref{mest} and  gathering Fubini's Theorem and
$$h(\tau)-h(\tau+t-s)=\int_\tau^{\tau+t-s}\frac{{\rm d}h}{{\rm d}t}(\sigma ){\rm d}\sigma,$$
the right-hand side of \eqref{mes} is smaller than
\begin{equation*}
\begin{split}
(t-s)\|A u_0-h(0)\|_{L^\infty(\Om)}+&\int_{0}^{t-s}\|h(0)-h(\tau)\|_{L^\infty(\Om)} {\rm d}\tau\\
&+ \int_{0}^{s}\|h(\tau)-h(\tau+t-s)\|_{L^\infty(\Om)} {\rm d}\tau.
\end{split}
\end{equation*}
Thus
\begin{equation}\label{mainestimate}
\begin{split}
\|u(t)-u(s)\|_{L^\infty(\Om)}\leq  (t-s)\|A u_0&-h(0)\|_{L^\infty(\Om)}\\
&+(t-s)\int_0^{T}\left \|\frac{{\rm d}h}{{\rm d}t}(\tau)\right\|_{L^\infty(\Om)} {\rm d}\tau.
\end{split}
\end{equation}
Dividing the expression \eqref{mainestimate} by $|t-s|$, 
we get that $u$ is a Lipschitz function  and since $\frac{\partial u}{\partial t}\in L^2(Q_T)$, passing to the limit $|t-s|\to 0$ we obtain that 
$\frac{u(t)-u(s)}{t-s}\to \frac{\partial u}{\partial t}$ as $s\to t$  weakly in $L^2(Q_T)$ and *-weakly in $L^\infty(Q_T)$. Furthermore,
\begin{equation*}
\left\|\frac{\partial u}{\partial t}\right\|_{L^\infty(\Om)}\leq \displaystyle\liminf_{s\to t}\frac{\|u(t)-u(s)\|_{L^\infty(\Om)}}{|t-s|}.
\end{equation*}
\medskip Therefore, we get $u\in W^{1,\infty}(0,T;L^\infty(\Omega))$ as well as inequality \eqref{estimationbiss2}.\qed

\noindent The proof of Theorem \ref{continuity u with f} follows easily:\\
\noindent{\bf Proof of Theorem \ref{continuity u with f}}:
The existence of mild solutions can be obtained similarly as in the proof of Theorem \ref{continuity de u} taking into account the $L^\infty$-bound given by the barrier functions $v_0$, $v_1$ and $v_2$.
{\it(i)} is the consequence of \eqref{estimation2} together with the fact that $f$ is locally Lipschitz and the Gronwall Lemma.

Regarding assertion {\it(ii)}, we follow the proof of Proposition \ref{continuity de u}: assume without loss of generality that $t>s$. Then, 
\begin{equation*}
\begin{split}
\|u(t)-u(s)\|_{L^\infty(\Omega)}&\leq \|u_0-u(t-s)\|_{L^\infty(\Omega)}\\
&+\int_0^s\|f(x,u(\tau))-f(x,u(\tau+t-s))\|_{L^\infty(\Omega)}{\rm d}\tau.
\end{split}
\end{equation*}
From assertion {\it (i)} and the fact that $f$ is Lipschitz on $[v_1(T), v_2(T)]$, it follows that
\begin{equation*}
\begin{aligned}
 \|u(t)\!-u(s)\|_{L^\infty(\Omega)}\!&\leq \|u_0\!-u(t-s)\|_{L^\infty(\Omega)}\!+\omega\!\!\int_0^s\!\!e^{\omega \tau}{\rm d}\tau \|u_0-u(t-s)\|_{L^\infty(\Omega)}\\
& \leq e^{\omega s}\|u_0-u(t-s)\|_{L^\infty(\Omega)}.
\end{aligned}
\end{equation*}
Now, we estimate the term $\|u_0-u(t-s)\|_{L^\infty(\Omega)}$ in the following way:
\begin{equation*}
\begin{split}
\|u_0-u(t-s)\|_{L^\infty(\Omega)}&\leq\int_0^{t-s}\|A u_0-f(x,u(\tau))\|_{L^\infty(\Omega)}{\rm d}\tau\\
&\leq (t-s)\|A u_0-f(x,u_0)\|_{L^\infty(\Omega)}\\
&+
\omega\int_0^{t-s}\|u_0-u(\tau)\|_{L^\infty(\Omega)}
{\rm d}\tau.
\end{split}
\end{equation*}
From Gronwall lemma, we deduce that
\begin{equation}
\nonumber
\|u_0-u(t-s)\|_{L^\infty(\Omega)}\leq(t-s)e^{\omega(t-s)}\|A u_0-f(x,u_0)\|_{L^\infty(\Omega)}.
\end{equation}
Gathering the above estimates, we get
\begin{equation*}
\|u(t)-u(s)\|_{L^\infty(\Omega)}\leq(t-s)e^{\omega t}\|A u_0-f(x,u_0)\|_{L^\infty(\Omega)}.
\end{equation*}
 Then, the rest of the proof follows with the same arguments as in the proof \medskip of Theorem \ref{continuity de u}.\qed

\noindent We are ready now to prove our stabilization result:\\
{\bf Proof of Theorem \ref{stabilization}}: From Proposition \ref{maccretive}, $A_f(u)=-\Delta_{p(x)}u-f(x,u)$ is m-accretive in $L^\infty(\Omega)$ and according to the remark \ref{omega-acretivity}, Theorem \ref{continuity u with f} holds with $u_0\in C^1_0(\overline{\Omega})$ replacing $A$ by $A_f$, the barrier functions  $v_0$, $v_1$ and $v_2$ by the subsolution $-w_{\mu}$ and the supersolution $w_{\mu}$ respectively and for $\mu>0$ large. Furthermore, $\omega=0$ in assertion {\it(i)} of Theorem \ref{continuity u with f} and the solution is global. Now, from the comparison principle, the solution $u(t)$ emanating from $u_0$
belongs to the conical shell $[u_1(t),u_2(t)]$ where $u_1$ and $u_2$ are the mild solutions with initial data $-w_{\mu}$ and $w_{\mu}\in {\mathcal D}(A_f)$, respectively. Again from the weak comparision principle, $t\to u_1(t)$ and $t\to u_2(t)$ are nondecreasing and nonincreasing respectively. Furthermore, from the uniqueness of the mild solution, $u_1$ and $u_2$ converge in $L^\infty(\Omega)$ to the stationary solution, $u_\infty$, to $(P_T)$ which is unique  from the monotonicity of $A_f$. Then, $u(t)\to u_\infty$ as $t\to\infty$.\qed

\begin{rem}
If one assumes in addition that $f(x,0)\geq 0$ for all $x\in \Omega$, then it is easy to prove that $u_\infty$ is positive and if $u_0$ is nonegative, $u(t)$ is nonnegative for all $t\geq 0$.
\end{rem}
\section*{Acknowledgement} We wish to thank Guy Vallet for his time and his advices. Our discussions about the problem always were very helpful and beneficial. The authors were funded by IFCAM (Indo-French Centre for Applied Mathematics) under the project ''Singular phenomena in reaction diffusion equations and in conservation laws''. 
\appendix
\section{Algebraic tools}
\noindent We recall suitable inequalities due to Simon \cite{S2}: for all $u,\ v\in \R^d$
\begin{equation}\label{A0}
\left||u|^{p-2}u-|v|^{p-2}v\right|\leq \left\{
\begin{array}{l l}
c |u-v|(|u|+|v|)^{p-2} & \mbox{if } p\geq 2;\\
c |u-v|^{p-1} &\mbox{if } p\leq 2;
\end{array}\right.
\end{equation}
\begin{equation}\label{A00}
<|u|^{p-2}u-|v|^{p-2}v, u-v>\geq \left\{
\begin{array}{l l}
\tilde c |u-v|^{p} & \mbox{if } p\geq 2;\\
\tilde c \displaystyle{\frac{|u-v|^{2}}{(|u|+|v|)^{2-p}} }&\mbox{if } p\leq 2
\end{array}\right.
\end{equation}
where $c,\ \tilde c$ are positive constants and $<.,.>$ is the scalar product of $\R^d$.
\begin{lem}\label{A6}
Let $p\in L^\infty(\Omega)$ such that $p\geq 0$, $p\not\equiv 0$. Let $q\in \mathcal P(\Omega)$ such that $p(x)q(x)\geq 1$
a.e. on $\Omega$. Then for every $f\in L^{p(x)q(x)}(\Omega)$,
\begin{equation}
\|f^{p(x)}\|_{L^{q(x)}(\Om)}\leq \|f\|^{p_-}_{L^{p(x)q(x)}(\Om)}+\|f\|^{p_+}_{L^{p(x)q(x)}(\Om)}.
\end{equation}
\end{lem}
\begin{proof}
To simplify the notations we set $\alpha(.)=p(.)q(.)$. Let $f\in L^{\alpha(x)}(\Omega)$, we define
$\beta =p_-$  if $\  \|f\|_{L^{\alpha(x)}(\Om)}\leq 1$ and $\beta=p_+$ if $\ \|f\|_{L^{\alpha(x)}(\Om)}>1$.
Then, for $\lambda =\|f\|_{L^{\alpha(x)}(\Om)}$
$$\rho_q\left(\frac{f^{p(x)}}{\lambda^\beta}\right)=\int_\Omega
\left|\frac{f}{\|f\|_{L^{\alpha(x)}(\Om)}}\right|^{\alpha(x)}.\frac{1}{\lambda^{\beta q(x)-\alpha(x)}}dx\leq 1 .$$
Hence by the definition of the norm of $L^{\alpha(x)}(\Om)$, we obtain the estimate.
\end{proof}
\noindent Now we prove a  technical inequality in the case $p_+\leq 2$.
\begin{lem}\label{A3} Let $p\in C(\overline \Omega)$, $1<p_-\leq p_+\leq 2$. Then, there exists $C>0$ such
that
for any $u,\, v\in \X$
\begin{equation}\label{eqA1}
\begin{split}
\int_\Omega (|\nabla u|^{p(x)-2}\nabla u&-|\nabla v|^{p(x)-2}\nabla v).\nabla(u-v)dx\\
&\geq C\left(\frac{\int_\Om
|\nabla(u-v)|^{p(x)}dx}{\|(|\nabla u|+|\nabla v|)^{\alpha(x)}\|_{L^{\frac{2}{2-p(x)}}(\Om)}}\right)^\gamma,
\end{split}
\end{equation}
where $\alpha(x)=\frac{p(x)(2-p(x))}{2}$ and $\gamma\in \{\frac{2}{p_+};\frac{2}{p_-}\}$.
\end{lem}
\begin{proof}
First, H\"older inequality implies
\begin{equation}\label{eqA2}
\begin{split}
\int_\Om |\nabla(u-v)|^{p(x)}dx &\left(\|(|\nabla u|+|\nabla v|)^{\alpha(x)}\|_{L^{\frac{2}{2-p(x)}}(\Om)}\right)^{-1}\\
&\leq
C\left\|\frac{|\nabla(u-v)|^{p(x)}}{(|\nabla u|+|\nabla v|)^{\alpha(x)}}\right\|_{L^{\frac{2}{p(x)}}(\Om)} \eqdef{C\mathcal I}.
\end{split}
\end{equation}
On the other hand, since $p_+\leq2$ we have using \eqref{A00}:
$$\int_\Omega (|\nabla u|^{p(x)-2}\nabla u-|\nabla v|^{p(x)-2}\nabla v).\nabla(u-v)dx\geq C\int_\Omega
\frac{|\nabla(u-v)|^{2}}{(|\nabla u|+|\nabla v|)^{2-p(x)}}dx.$$
In the case where $\mathcal I<1$, plugging the last inequality \eqref{eqA2} and \eqref{b} we
obtain \eqref{eqA1} with $\gamma=\frac{2}{p_-}$. In the other case: $\mathcal I\geq 1$ then by \eqref{a}, we get inequality \eqref{eqA1} with
$\gamma=\frac{2}{p_+}$.
\end{proof}
\begin{rem}\label{rA1} In the case $p_-\geq 2$, using inequalities \eqref{A00} , on can be easily prove that that there exists
$\tilde C>0$ such that:
$$\int_\Omega (|\nabla u|^{p(x)-2}\nabla u-|\nabla v|^{p(x)-2}\nabla v).\nabla(u-v)dx\geq \tilde C\int_\Om
|\nabla(u-v)|^{p(x)}dx.$$
\end{rem}
\noindent We have the following comparison principle:
\begin{lem}\label{wc}
Let $u,v\in\X$ such that $-\Delta _{p(x)}u  \geq -\Delta _{p(x)}v$ in $\Om$ in the sense that
$$ \forall \varphi\in \X,\ \varphi\geq0,\quad \int_{\Omega}(|\nabla u|^{p(x)-2}\nabla u-|\nabla v|^{p(x)-2}\nabla v)\cdot\nabla\varphi\, dx\geq 0.$$
Then $u\geq v$ {\it a.e.} in $\Om$.
\end{lem}
\begin{proof}
Set $\varphi=\max(v-u,0)\in \X$. Then
\begin{equation*}
\int_{\{v> u\}}(|\nabla u|^{p(x)-2}\nabla u-|\nabla v|^{p(x)-2}\nabla v)\cdot\nabla\varphi\, dx\geq 0
\end{equation*}
By Lemma \ref{A3} and Remark \ref{rA1} we deduce that
$$\int_{\{v> u\}}|\nabla \varphi|^{p(x)}dx\leq 0.$$
Hence $\varphi=0$ a.e. in $\Om$. Therefore $u\geq v$ a.e. in $\Om$.
\end{proof}
\section{Convergence tools}
\begin{lem}\label{A4}
Let $(u_n)$ be a bounded sequence in $L^\infty(0,T,\X)$ uniformly in $n$. Let $u\in L^\infty(0,T,\X)$ such that
\[\lim_{n\rightarrow +\infty} \ii (|\nabla u_n|^{p(x)-2}\nabla u_n-|\nabla u|^{p(x)-2}\nabla u).\nabla
(u_n-u)dxdt=0.\]
Then, $\nabla u_n$ converges to $\nabla u$ in $L^{p(x)}(Q_T)$.
\end{lem}
\begin{proof}
Define $\nabla X=|\nabla u_n|^{p(x)-2}\nabla u_n-|\nabla u|^{p(x)-2}\nabla u$. Thus we have,
\begin{equation*}
\begin{split}
-\int_0^T\langle\Delta_{p(x)}u_n-\Delta_{p(x)}u,u_n-u\rangle dt&=\int_0^T\!\!\!\int_{\Omega\backslash \Omega_2} \nabla
X.\nabla(u_n-u)dxdt\\
&+\int_0^T\!\!\!\int_{\Omega_2} \nabla X.\nabla(u_n-u)dxdt
\end{split}
\end{equation*}
where $\Omega_2=\{p(x)>2\}$. By Lemma \ref{A3}, with $\alpha(x)=\frac{p(x)(2-p(x))}{2}$, and Remark \ref{rA1}, we get as $n\rightarrow +\infty$
$$\int_0^T\!\!\!\int_{\Om_2} |\nabla(u_n-u)|^{p(x)}dxdt \rightarrow 0$$
and
$$\int_0^T\frac{\int_{\Om\backslash \Om_2}|\nabla(u_n-u)|^{p(x)}dx}{\|(|\nabla u_n|+|\nabla
u|)^{\alpha(x)}\|_{L^{\frac{2}{2-p(x)}}(\Om\backslash \Om_2)}}dt\rightarrow 0.$$
Applying Lemma \ref{A6}, we prove that the mapping $$t\rightarrow \|(|\nabla u_n|+|\nabla
u|)^{\alpha(x)}\|_{L^{\frac{2}{2-p(x)}}(\Om\backslash \Om_2)}$$ is bounded on $[0,T]$.
Hence we obtain
$$\int_0^T\!\!\!\int_{\Om\backslash \Om_2} |\nabla(u_n-u)|^{p(x)}dxdt \rightarrow 0.$$
We conclude by Proposition \ref{P1} {\it(\,ii\,)}  that $\nabla u_n$ converges to $\nabla u$ in $L^{p(x)}(Q_T)$.
\end{proof}
\noindent Finally we recall the following Corollary A.3 in \cite{LU}.
\begin{lem}\label{A5}
Let $\Om$ be a smooth bounded domain. Consider $(u_n)$, $u\in L^{p(x)}(\Om)$ such that $u_n$ converges to $u$ weakly in $L^{p(x)}(\Om)$. Then, $\rho_p(u_n)$ converges to $\rho_p(u)$
implies that $u_n$ converges to $u$ in $L^{p(x)}(\Om)$.
\end{lem}

\section{Regularity result}\label{ApC}
\noindent We begin by recalling the regularity result due to Fan and Zhao \cite{FZ}:
\begin{prop}[Theorem 4.1 in \cite{FZ}]\label{reg0}
Let $p\in C(\overline\Omega)$ and $u\in \X$ satisfying
\begin{equation*}
\int_\Omega |\nabla u|^{p(x)-2}\nabla u .\nabla \Psi dx=\int_\Omega f(x,u) \Psi dx, \quad \forall \Psi \in \X,
\end{equation*}
where $f$ satisfies for all $(x,t)\in \Omega\times\R$, $|f(x,t)|\leq c_1+c_2|t|^{r(x)-1}$ with $r\in C(\overline\Om)$ and $\forall x\in\overline \Om$, $1<r(x)<p^*(x)$. Then $u\in L^\infty(\Omega)$.
\end{prop}
\noindent For $f(x,.)=f(x)$, we have the following proposition.
\begin{prop}\label{rg1}
Let $p\in C(\bar\Omega)$ with $p^-<d$ and $u\in \X$ satisfying
\begin{equation}\label{ws}
\int_\Omega |\nabla u|^{p(x)-2}\nabla u .\nabla \Psi dx=\int_\Omega f \Psi dx, \quad \forall \Psi \in \X,
\end{equation}
where $f\in L^{q}(\Omega)$, $q>\frac{d}{p_-}$. Then $u\in L^\infty(\Omega)$.
\end{prop}
\noi To prove Proposition \ref{rg1}, we need a regularity lemma.  
\begin{lem}\label{fsb1}
Let $u\in W_0^{1,p}(\Omega)$, $1<p<d$, satisfying for any  $B_R$, $R<R_0$, and for all $\sigma\in (0,1)$, and
any $k\geq k_0>0$
\begin{equation*}\label{equa}
\begin{split}
\int_{A_{k,\sigma R}} |\nabla u|^p dx &\leq C\left[ \int_{A_{k, R}}
\left|\frac{u-k}{R(1-\sigma)}\right|^{p^*}dx+k^\alpha|A_{k,R}|+ |A_{k,R}|^{\frac{p}{p^*}+\varepsilon}\right.\\
&\left.+\left(\int_{A_{k, R}} \left|\frac{u-k}{R(1-\sigma)}\right|^{p^*}dx\right)^{\frac{p}{p^*}}|A_{k,R}|^\delta\right]
\end{split}
\end{equation*}
where $A_{k,R}=\{x\in B_R\cap \Omega\ |\ u(x)>k\}$, $0<\alpha<p^*=\frac{dp}{d-p}$ and $\varepsilon,\,\delta>0$. Then $u\in L^\infty(\Omega)$.
\end{lem}

\begin{proof}
Fusco and Sbordone have already proved in \cite{FS} the local boundedness of $u$ in the case $u\in W^{1,p}(\Omega)$ and the inequality is satisfied for any $B_R\subset \subset \Omega$. We claim that the result is still valid in our situation. For that we will prove the boundednes of $u$ in a neighborhood of the boundary $\partial\Omega$.\\
Let $x_0\in \partial \Omega$, $B_R$ the ball centred in $x_0$. We define $K_R\eqdef B_R\cap \Omega$ and we set 
$$r_h=\frac{R}{2}+\frac{R}{2^{h+1}}, \quad \tilde r_h=\frac{r_h+r_{h+1}}{2} \ \mbox{ and } \ k_h=k\left(1-\frac{1}{2^{h+1}} \right) \quad \mbox{for any $h\in \N$}.$$
Define also
$$I_h=\int_{A_{k_h,r_h}} |u(x)-k_h|^{p^*}dx \quad \mbox{and} \quad \varphi(t)=\left\{
\begin{array}{l l}
1 & \mbox{if}\ 0\leq t\leq \frac12,\\
0 & \mbox{if}\ t\geq \frac34
\end{array}
\right.
$$
satisfying $\varphi\in C^1([0,+\infty);[0,1])$. We set $\varphi_h(x)=\varphi\left(\frac{2^{h+1}}{R}(|x|-\frac{R}{2})\right)$. Hence $\varphi_h=1$ on $B_{r_{h+1}}$ and $\varphi_h=0$ on $\R^d\backslash B_{\tilde r_{h+1}}$.\\
We have
\begin{equation*}
\begin{split}
I_{h+1}&=\int_{A_{k_{h+1},r_{h+1}}} |u(x)-k_{h+1}|^{p^*}dx=\int_{A_{k_{h+1},r_{h+1}}} (|u(x)-k_{h+1}|\varphi_h(x))^{p^*}dx\\
&\leq \int_{K_R} ((u(x)-k_{h+1})^+\varphi_h(x))^{p^*}dx.
\end{split}
\end{equation*}
Since $u\in W^{1,p}_0(\Omega)$, $(u-k_{h+1})^+\varphi_h\in W^{1,p}_0(K_R)$. Thus
\begin{equation*}
\begin{split}
I_{h+1}&\lesssim \left(\int_{K_R} |\nabla((u-k_{h+1})^+\varphi_h)|^{p}dx\right)^{\frac{p^*}{p}}\\
&\lesssim \left(\int_{A_{k_{h+1},\tilde r_{h}}} |\nabla u|^{p}dx+\int_{A_{k_{h+1},\tilde r_{h}}}(u-k_{h+1})^{p}dx\right)^{\frac{p^*}{p}}
\end{split}
\end{equation*}
where we use the notation $f\lesssim g$ in the sense there exists a constant $c>0$ such that $f\leq c g$. Since $\tilde r_h<r_h$, we have
\begin{equation*}
\begin{split}
I_{h+1}^{\frac{p}{p^*}}&\lesssim 2^{hp^*}\int_{A_{k_{h+1},r_{h}}}|u-k_{h+1}|^{p^*}dx +k_{h+1}^\alpha|A_{k_{h+1},r_{h}}|+|A_{k_{h+1},r_{h}}|^{\frac{p}{p^*}+\varepsilon}\\
&+2^{hp}\left(\int_{A_{k_{h+1},r_{h}}}|u-k_{h+1}|^{p^*}dx\right)^{\frac{p}{p^*}}|A_{k_{h+1},r_{h}}|^\delta\\
&+{\int_{A_{k_{h+1},r_{h}}}|u-k_{h+1}|^{p^*}dx}.
\end{split}
\end{equation*}
Moreover, for any $h$, $k_h\leq k_{h+1}$, this implies
\begin{equation}\label{Jh1}
\begin{split}
I_h&=\int_{A_{k_{h},r_{h}}}|u-k_{h}|^{p^*}dx \geq \int_{A_{k_{h+1},r_{h}}}|u-k_{h}|^{p^*}dx\\
&\geq \int_{A_{k_{h+1},r_{h}}}|k_h-k_{h+1}|^{p^*}dx=|A_{k_{h+1},r_{h}}||k_{h+1}-k_h|^{p^*}.
\end{split}
\end{equation}
Then, for any $k>k_0$ and $h\in \N$
$$  |A_{k_{h+1},r_{h}}|+k_{h+1}^\alpha|A_{k_{h+1},r_{h}}|\lesssim 2^{hp^*}I_h$$
where the constant in the notation depends only on $k_0$, $p$ and $\alpha$. Replacing in \eqref{Jh1}, we obtain
\begin{equation}\label{Jh2}
I_{h+1}^{\frac{p}{p^*}} \lesssim 2^{hp^*}I_h+ 2^{h(p+\varepsilon p^*)}I_h^{\frac{p}{p^*}+\varepsilon}+2^{h(p+\delta p^*)}I_h^{\frac{p}{p^*}+\delta}.
\end{equation}
Setting $M=\frac{p}{p^*}\max(p^*,\, p+\varepsilon p^*,\, p+\delta p^*)$ and $\theta=\min(1-\frac{p}{p^*},\,\varepsilon,\,\delta)$ and noting 
$$I_h\leq \int_{K_R} (|u-k_h|^+)^{p^*}dx\leq \int_{K_R}|u|^{p*}\leq \|u\|_{W^{1,p}_0}^{p*}.$$
Hence \eqref{Jh2} becomes
\begin{equation*}
I_{h+1}\lesssim 2^{hM}I_h^{1+\frac{\theta p^*}{p}}
\end{equation*}
where the constant depends on $||u||_{W^{1,p}_0}$, $k_0$, $\alpha$ and $p$. We need the following lemma to conclude.
\begin{lem}[Lemma 4.7, Chapter 2, \cite{LU}]
Let $(x_n)$ be a sequence such that $x_0\leq \lambda^{-\frac1\eta}\mu^{-\frac{1}{\eta^2}}$ and $x_{n+1}\leq \lambda\mu^nx_n^{1+\eta}$, for any $n\in \N^*$ with $\lambda$, $\eta$ and $\mu$ are positive constants and $\mu>1$. Then $(x_n)$ converges to $0$ as $n\rightarrow +\infty$.
\end{lem}
\noindent It suffices to prove that $I_0$ is small enough. Indeed $u\in L^{p^*}(\Omega)$ implies
$$I_0=\int_{A_{\frac{k}{2},R}} |u-\frac{k}{2}|^{p^*}dx \rightarrow 0 \quad\mbox{ as }\quad k\rightarrow \infty.$$
Hence for $k$ large enough, $I_0\leq C^{-\frac1\eta}(2^M)^{-\frac{1}{\eta^2}}$ with $\eta=\frac{\theta p^*}{p}$. Thus $I_h$ converges to $0$ as $h\rightarrow +\infty$ and 
$$\int_{A_{k,\frac{R}{2}}} |u-k|^{p^*}dx=0.$$
We deduce that $u\leq k$ on $K_{\frac{R}{2}}$. In the same way, we prove that $-u\leq k$ on $K_{\frac{R}{2}}$. Since $\overline \Omega$ is compact, we conclude that $u\in L^\infty(\Omega)$.
\end{proof}

\noindent \textbf{Proof of Proposition \ref{rg1}:} We follow the idea of the proof of Theorem 4.1 in \cite{FZ}.\\
Let $x_0\in \overline\Omega$, $B_{R}$ the ball of radius $R$ centered in $x_0$ and $K_R= \Omega\cap B_R$. We define
$$p^+\eqdef \max_{K_{R}} p(x) \quad \mbox{and} \quad p^-\eqdef \min_{K_{R}} p(x)$$
and we choose $R$ small enough such that $p^+< (p^-)^*=\frac{dp^-}{d-p^-}$. \\
Fix $(s,t)\in (\R_+^*)^2$, $t<s<R$ then  $K_t\subset K_s\subset K_R$. Define $\varphi\in
C^\infty(\Omega)$, $0\leq \varphi\leq 1$ such that
$$\varphi=\left\{\begin{array}{l l}
1 & \mbox{in } B_t,\\
0 & \mbox{in }\R^d\backslash B_s
\end{array}\right.$$
{satisfying $|\nabla\varphi|\lesssim 1/(s-t)$}. Let $k\geq 1$, using the same notations as previously $A_{k,\lambda}=\{y\in K_\lambda \ | \ u(y)>k\}$ and taking $\Psi=\varphi^{p^+}(u-k)^+\in \X$ in \eqref{ws}, we obtain
\begin{equation}\label{B3}
\begin{split}
\int_{A_{k,s}} |\nabla u|^{p(x)}\varphi^{p^+}\,dx&+p^+\int_{A_{k,s}} |\nabla u|^{p(x)-2}\nabla u\cdot\nabla \varphi
\varphi^{p^+-1}(u-k)^+\,dx\\
&=\int_{A_{k,s}} f\varphi^{p^+}(u-k)\,dx.
\end{split}
\end{equation}
Hence by Young inequality, for $\epsilon>0$, we have
\begin{equation*}
\begin{split}
p^+\int_{A_{k,s}} |\nabla u|^{p(x)-2}\nabla u.\nabla \varphi \varphi^{p^+-1}(u-k)\,dx& \\ 
\leq \varepsilon \int_{A_{k,s}} |\nabla u|^{p(x)}\varphi^{(p^+-1)\frac{p(x)}{p(x)-1}}\,dx
+c\varepsilon^{-1}&\int_{A_{k,s}} (u-k)^{p(x)}|\nabla\varphi|^{p(x)}\,dx.
\end{split}
\end{equation*}
Since $|\nabla\varphi|\leq c/(s-t)$ and for any $x\in K_R$, $p^+\leq (p^+-1)\frac{p(x)}{p(x)-1}$, we have
$\varphi^{(p^+-1)\frac{p(x)}{p(x)-1}}\leq \varphi^{p^+}$. This implies
\begin{equation}\label{est1}
\begin{split}
p^+\int_{A_{k,s}} |\nabla u|^{p(x)-2}\nabla u.\nabla \varphi \varphi^{p^+-1}(u-k)\,dx&\\
\leq \varepsilon \int_{A_{k,s}} |\nabla u|^{p(x)}\varphi^{p^+}\,dx 
+ c\varepsilon^{-1}&\int_{A_{k,s}} \left(\frac{u-k}{s-t}\right)^{p(x)}\,dx.
\end{split}
\end{equation}
Using H\"older inequality we estimate the right-hand side of \eqref{B3} as follows:
\begin{equation*}
\int_{A_{k,s}} f\varphi^{p^+}(u-k)\,dx\leq \|f\|_{L^q}\left(\int_{A_{k,s}}
(u-k)^{\frac{q}{q-1}}\,dx\right)^{\frac{q-1}{q}}.
\end{equation*}
Since $q>\frac{d}{p^-}$, we have $\frac{(p^-)^*}{p^-}\frac{q-1}{q}>1$. So, applying once again the H\"older inequality, we
obtain
\begin{equation}\label{est2}
\int_{A_{k,s}} f\varphi^{p^+}(u-k)\,dx\leq C \left(\int_{A_{k,s}}
(u-k)^{\frac{(p^-)^*}{p^-}}\,dx\right)^{\frac{p^-}{(p^-)^*}}|A_{k,s}|^\delta,
\end{equation}
where $\delta= \frac{q-1}{q}-\frac{p^-}{(p^-)^*}>0$. Set $\{u-k>s-t\}=\{x\in K_R\ |\ u(x)-k>s-t\}$ and
its complement as $\{u-k\leq s-t\}$. Now we split the integral in the right-hand side of \eqref{est2} on $\Theta=A_{k,s}\cap\{u-k>s-t\}$ and $A_{k,s}\backslash \Theta$:
\begin{equation}\label{est3}
\begin{split}
\mathcal I\eqdef &\int_{A_{k,s}} \left(\frac{u-k}{s-t}\right)^{(p^-)^*}\,dx+|A_{k,s}| \\
&\gtrsim \int_{\Theta} \left(\frac{u-k}{s-t}\right)^{\frac{(p^-)^*}{p^-}}(s-t)^{\frac{(p^-)^*}{p^-}}\,dx\\ &+ 
\int_{A_{k,s}\backslash \Theta} \left(\frac{u-k}{s-t}\right)^{\frac{(p^-)^*}{p^-}}(s-t)^{\frac{(p^-)^*}{p^-}}\,dx.
\end{split}
\end{equation}
In the same way, the second term in the right-hand side of \eqref{est1} can be estimated as follows.
\begin{equation}\label{est4}
\int_{\Theta} \left(\frac{u-k}{s-t}\right)^{p(x)}\,dx +\int_{A_{k,s}\backslash \Theta}
\left(\frac{u-k}{s-t}\right)^{p(x)}\,dx \lesssim \mathcal I.
\end{equation}
Finally, plugging \eqref{est1}-\eqref{est4} and we obtain for $\varepsilon$ small enough
\begin{equation*}
\int_{A_{k,s}} |\nabla u|^{p(x)}\varphi^{p^+}\,dx\lesssim  \mathcal I + |A_{k,s}|^\delta\mathcal I^{\frac{p^-}{(p^-)^*}}
\end{equation*}
where the constant depends on $p,\,R$ and $\varepsilon$. Moreover we have
$$\mathcal I^{\frac{p^-}{(p^-)^*}}\lesssim \left(\int_{A_{k,s}}
\left(\frac{u-k}{s-t}\right)^{(p^-)^*}\,dx\right)^{\frac{p^-}{(p^-)^*}}+|A_{k,s}|^{\frac{p^-}{(p^-)^*}}.$$
Hence using the Young inequality, we obtain the following estimate.
\begin{equation*}
\begin{split}
\int_{A_{k,t}}|\nabla u|^{p^-}\,dx\leq&\int_{A_{k,s}} |\nabla u|^{p(x)}\varphi^{p^+}\,dx+|A_{k,s}|\\
\lesssim& \int_{A_{k,s}} \left(\frac{u-k}{s-t}\right)^{(p^-)^*}\,dx +|A_{k,s}|^{\frac{p^-}{(p^-)^*}+\delta}\\
&+|A_{k,s}|+|A_{k,s}|^\delta\left(\int_{A_{k,s}}
\left(\frac{u-k}{s-t}\right)^{(p^-)^*}\!\!\!dx\right)^{\frac{p^-}{(p^-)^*}}.
\end{split}
\end{equation*}
\medskip
By Lemma \ref{fsb1}, we deduce that $u$ bounded in $\Omega$. \qed

\noindent Combining Proposition \ref{reg0} and \ref{rg1}, we have the following corollary:
\begin{cor}\label{reg2}
Let $p\in C(\bar\Omega)$ such that $p^-<d$ and $u\in W^{1,p(x)}_0(\Omega)$ satisfying
\begin{equation*}
\int_\Omega |\nabla u|^{p(x)-2}\nabla u \cdot\nabla \Psi dx=\int_\Omega(f(x,u)+g) \Psi dx, \quad \forall \Psi \in \
\X,
\end{equation*}
where $f$ satisfies $|f(x,t)|\leq c_1+c_2|t|^{r(x)-1}$ with $r\in C(\overline\Om)$ and $\forall x\in\overline \Om$, $1<r(x)<p^*(x)$ and $g\in L^{q}$, $q>\frac{d}{p_-}$. Then $u\in L^\infty(\Omega)$.
\end{cor}

\end{document}